\documentclass[a4paper,11pt,reqno]{amsart}

\usepackage{url}
\usepackage[colorlinks=true, linkcolor=blue, citecolor=ForestGreen, urlcolor=blue]{hyperref}
\usepackage{cite}

\usepackage{latexsym,amsmath,amssymb,amsfonts,amsthm}
\usepackage{mathrsfs}
\usepackage{mathtools}
\usepackage{dsfont}
\usepackage{bbm}
\usepackage{soul}
\usepackage{cancel}
\usepackage[utf8]{inputenc}

\usepackage[dvipsnames]{color}
\usepackage[dvipsnames, cmyk]{xcolor}
\usepackage{graphicx}
\usepackage{multirow}
\usepackage{enumerate}
\usepackage{enumitem}
\usepackage{subcaption}
\usepackage{comment}

\usepackage{a4wide}

\allowdisplaybreaks
\numberwithin{equation}{section}

\definecolor{grey}{rgb}{.7,.7,.7}
\definecolor{refkey}{gray}{.45}
\definecolor{labelkey}{gray}{.45}

\newtheorem{theorem}{Theorem}[section]
\newtheorem{proposition}[theorem]{Proposition}
\newtheorem{lemma}[theorem]{Lemma}

\theoremstyle{remark}
\newtheorem{remark}[theorem]{Remark}
\theoremstyle{definition}

\newtheorem{example}[theorem]{Example}

\usepackage{tikz}
\usepackage{pgf,pgfplots}
\pgfplotsset{compat=newest}
\usetikzlibrary{calc}
\usetikzlibrary{decorations.markings}
\usetikzlibrary{decorations.pathmorphing}
\usetikzlibrary{decorations.shapes}
\usetikzlibrary{shapes,arrows,shapes.geometric,patterns,fadings}
\usetikzlibrary{pgfplots.fillbetween}

\renewcommand{\bar}[1]{\overline{#1}}

\renewcommand{\div}{\operatorname{div}}

\newcommand{\Id}{\operatorname{Id}}

\let\epsilon\varepsilon

\edef\restoreparindent{\parindent=\the\parindent\relax}
\usepackage{parskip}
\restoreparindent
\usepackage{stmaryrd}

\newcommand{\lbr}{\llbracket}
\newcommand{\rbr}{\rrbracket}

\newcommand{\e}{\varepsilon}

\newcommand{\N}{\mathbb N}

\newcommand{\R}{\mathbb R}

\newcommand{\dd}{\,\mathrm{d}}

\renewcommand{\setminus}{\backslash}

\newcommand{\defeq}{\coloneqq}

\newcommand{\asymm}[2]{#1\triangle #2}

\newcommand{\ene}{\mathcal{E}_{\alpha,\beta}}
\newcommand{\enf}{\mathcal{F}_{\alpha}}
\newcommand{\eng}{\mathcal{G}_{\alpha}}

\newcommand{\nl}{\mathcal{N}_{\alpha}}
\newcommand{\disp}{\displaystyle}

\newcommand{\ba}{\begin{array}}
\newcommand{\ea}{\end{array}}
\newcommand{\trns}[1]{\widehat{#1}\,}

\newcommand{\bthm}{\begin{theorem}}
\newcommand{\ethm}{\end{theorem}}
\newcommand{\bprop}{\begin{proposition}}
\newcommand{\eprop}{\end{proposition}}
\newcommand{\blemma}{\begin{lemma}}
\newcommand{\elemma}{\end{lemma}}
\newcommand{\bexmpl}{\begin{example}}
\newcommand{\eexmpl}{\end{example}}

\newcommand{\beqn}{\begin{equation}}
\newcommand{\eeqn}{\end{equation}}
\newcommand{\beqns}{\begin{equation*}}
\newcommand{\eeqns}{\end{equation*}}

\newcommand{\pt}{\partial}

\newcommand{\Htwo}{\mathcal{H}^2}

\renewcommand{\leq}{\leqslant}
\renewcommand{\geq}{\geqslant}

\definecolor{mygreen}{rgb}{0.1,0.75,0.2}

\newcommand{\calF}{\mathcal{F}}
\newcommand{\calD}{\mathcal{D}}
\newcommand{\calP}{\mathcal{P}}
\newcommand{\calN}{\mathcal{N}}
\newcommand{\calE}{\mathcal{E}}

\newcommand{\Pe}{\mathcal{P}}

\newcounter{myenumi}
\DeclareMathOperator{\dive}{div}
\DeclareMathOperator{\Lip}{Lip}

\DeclareMathOperator{\loc}{loc}

\title[The liquid drop model with a Yukawa potential]{The liquid drop model with a Yukawa potential: existence of minimizers and sharp stability of the ball}

\author{Lia Bronsard}
\address[Lia Bronsard]{Department of Mathematics \& Statistics, McMaster University, Hamilton ON, Canada}
\email{bronsard@mcmaster.edu}

 \author{Kenneth DeMason}
 \address[Kenneth DeMason]{Department of Mathematics \& Statistics, McMaster University, Hamilton ON, Canada}
 \email{demasonk@mcmaster.edu}

\author{Ihsan Topaloglu}
\address[Ihsan Topaloglu]{Department of Mathematics and Applied Mathematics, Virginia Commonwealth University, Richmond, VA, USA}
\email{iatopaloglu@vcu.edu}

\date{\today}       


\begin{document}

\begin{abstract}
We study a long-range perturbation of the perimeter functional by a nonlocal repulsive term defined through the Yukawa kernel, minimized under a volume constraint. Because the kernel decays exponentially, the competition between surface tension and repulsion is governed by two independent quantities, the screening rate and the volume. Specifically we prove that (i) above an explicit critical screening rate, minimizers exist at every volume, whereas earlier work required the volume to be large; (ii) below an explicit volume threshold minimizers exist for every screening rate; (iii) at small volume, and uniformly in the screening rate, the ball is the unique minimizer up to translation; and (iv) there exists a sharp volume threshold, depending on the screening rate, for the ball to be a stable volume-constrained critical point. The threshold is given in closed form for every screening rate and reduces to the known unscreened value when $\alpha=0$. Our proofs overcome new technical changes due to the lack of homogeneity in the Yukawa kernel.
\end{abstract}

\maketitle



\section{Introduction}\label{sec:intro}

In this paper we are interested in the stability and global/local minimality of balls $B_r$ for the energy functional
	\[
		\enf(E) \defeq \Pe(E) + \int_E \int_E \frac{e^{-\alpha|x-y|}}{|x-y|}\dd x \dd y
	\]
in the class of sets of finite perimeter $E\subset \R^3$ such that $|E|=m$. Here $\alpha,m>0$ and $\Pe(E)$ denotes the perimeter of the set $E$ in the De Giorgi sense, see for instance \cite[Chapter 12]{Mag}. The nonlocal term is given by
	\begin{equation}\label{eqn: Nonlocal Definition}
		\nl(E) \defeq \int_E \int_E Y_\alpha(x-y) \dd x \dd y, \qquad  Y_\alpha(z) \defeq \frac{e^{-\alpha|z|}}{|z|}
	\end{equation}
where $Y_{\alpha}$ is the Yukawa kernel.

The kernel $Y_\alpha$ arises in models where a charged phase carrying an interfacial tension is immersed in a medium of mobile opposite charges. The mobile charges rearrange to cancel the charge they surround, and to first order this adds a restoring term to the electrostatic equation: the potential solves $(-\Delta+\alpha^{2})v=4\pi\rho$ rather than $-\Delta v=4\pi\rho$. Its Green function is $(4\pi)^{-1}Y_\alpha$, so pairwise interactions decay exponentially beyond the length $\ell=1/\alpha$ while remaining Coulombic below it. For example, in the inner crust of a neutron star the medium is a gas of degenerate electrons, in which protons form clusters surrounded by a dilute neutron gas. The Thomas--Fermi linearization of the electron response makes $\ell$ the electron Thomas--Fermi length, and two protons at distance $r$ then interact through $q^{2}e^{-r/\ell}/r=q^{2}Y_\alpha(r)$, with $q$ being the proton charge
(see for example \cite{CaplanHorowitz,ChugunovHorowitz}). Minimizing the associated free energy predicts the cluster shape, which as the density increases passes from near-spherical nuclei through rods, slabs, tubes and bubbles. This sequence of phases is referred to as nuclear pasta (see \cite{RPW,HSY,PR}). A second setting in which functionals of such form appear is regarding models for a charged phase in an electrolyte. There the medium consists of dissolved ions and the linearization is the Debye--H\"uckel approximation. In these models $\ell$ is the Debye length, set by the ionic strength, and it is adjustable by adding salt into the mixture. The screened repulsion obtained in this way is one of the two interactions in the classical theory of colloidal stability of Derjaguin, Landau, Verwey and Overbeek, in which it competes with a van der Waals attraction to determine whether suspended particles aggregate (see \cite{Isr,VO}).

When $\alpha=0$ the minimization problem with $Y_0(z)=|z|^{-1}$ among $E\subset \mathbb{R}^3$ with fixed mass is Gamow's liquid drop model of the atomic nucleus, and it has been long conjectured that minimizers are balls below a critical mass and that no minimizers exist above this threshold (see \cite{CMT} for a review on the topic). This conjecture has been resolved only very recently (see \cite{CG}). Screening the interaction changes this picture qualitatively, and the resulting minimization problem with $Y_{\alpha}$ is the object of very recent activity (cf. \cite{BMP,MNS}). 

The first author, Merlet, and Pegon \cite{BMP} have shown that the conjecture \emph{fails} when the interaction is screened: for suitable values of the screening parameter they show that there exists a connected minimizer in $\R^3$ which is not a ball, and balls fail to minimize at that mass. Their setting is exactly the one studied here, their Yukawa potential $Y_{1,\kappa}(x)=e^{-|x|/\kappa}/|x|$ corresponds to our $Y_\alpha$ with $\kappa=1/\alpha$, so that their functional is $\calF_{1/\kappa}$. In two dimensions, Muratov, Novaga and Simon \cite{MNS} analyze the same Yukawa profile in the large-mass limit for a certain regime of screening parameters, where the nonlocal term concentrates on the boundary and cancels the perimeter to leading order. They show that the next-order $\Gamma$-limit is a weighted sum of the perimeter and Euler's elastica functional evaluated on the system of boundary curves, and the minimizers are disks above a critical screening rate and rings below it. For a general class of integrable kernels, Pegon \cite{Peg} proves the existence of minimizers above an explicit critical screening rate, provided the mass is large enough.

In this paper we first establish the existence of minimizers of $\enf$ for any volume $m$ when the screening parameter is above a critical threshold. We follow the strategy in \cite{MNS} and extend the existence result in \cite{Peg} to the full range of $m$.
	\begin{theorem}\label{thm:exist_every_m}
		For $\alpha>(2\pi)^{1/3}$ and for any $m>0$, the problem 	
		\begin{equation}\label{eqn:energy_alpha_min}
		e_{\alpha}(m) \defeq \inf \Big\{ \enf(E) \, : \, |E|=m \Big\}
	\end{equation}
 admits a solution. Moreover $m \mapsto m^{-1}e_\alpha(m)$ is strictly decreasing.
	\end{theorem}
	
Our second theorem shows the existence of minimizers for any $\alpha$ provided $m$ is sufficiently small. Using the arguments of Frank and Nam \cite{FN} for the liquid drop model ($\alpha=0$ case), we prove the following.
	\begin{theorem}\label{thm:exist_every_alpha}
		Let
			\[
				\bar{m} \defeq \frac{5(2^{1/3}-1)}{2(1-2^{-2/3})}.
			\] 
		Then for any $\alpha\geq 0$ and $m\in(0,\bar{m})$, the problem \eqref{eqn:energy_alpha_min} admits a solution.
	\end{theorem}

Next we turn to characterizing minimizers. Arguing as in Carazzato, Fusco, and Pratelli \cite{CFP}, we prove that for sufficiently small volume, the global minimizer is the unit ball.
	\begin{theorem}\label{thm:global_min}
		There exists $m_0>0$ (independent of $\alpha$) such that for any $\alpha\geq 0$ and $m<m_0$, the unique minimizer, up to translations, of $\enf$ over sets $|E|=m$ is the ball $B[m]$ of volume $m$. 
	\end{theorem}
	
Finally, we derive a critical stability threshold for balls of volume $m$. 
	\begin{theorem}\label{thm: Stable Critical Threshold} For any $\alpha\geq 0$ there exists a computable $m_*(\alpha)>0$ such that if $m>0$ the ball $B[m]$ of volume $m$ is a volume-constrained stable critical point of $\enf$ if and only if $m \leq m_*(\alpha)$. In particular, when $\alpha \geq (2\pi)^{1/3}$ we have $m_*(\alpha)=+\infty$, so that $B[m]$ is always stable; when $\alpha <(2\pi)^{1/3}$ we have $m_*(\alpha)\geq 5$.
	\end{theorem}
As a consequence of this stability result, using the well-established scheme in the literature (see for instance \cite{AFM,BC,FFMMM}), we obtain the $L^1$-local minimality of $B[m]$. We note that our proof does not require any homogeneity, in contrast to \cite{BC,FFMMM}.
	\begin{theorem}\label{thm: L1 Local Minimality} For every $\alpha \geq 0$ and $m < m_*(\alpha)$ there exists $\epsilon_*(\alpha,m)>0$ such that if $|E|=m$ and $|E\Delta B[m]| \leq \epsilon_*m$ then
\[\enf(B[m])\leq \enf(E)\]
with equality if and only if $|E\Delta B[m]|=0$ up to translation. If instead $m>m_*(\alpha)$ then there exists a sequence of sets of finite perimeter $\{E_i\}_{i=1}^{\infty}$ with $|E_i|=m$ and $|E_i\Delta B[m]|\to 0$ but $\enf(B[m])>\enf(E_i)$ for every $i\in \mathbb{N}$.
	\end{theorem}
	
\begin{remark}[Relationship between variants of $\enf$]\label{rem:master}
One could consider the a priori more general three-parameter family of functionals
	\[
  		\ene(E)\defeq \Pe(E)+\beta\,\nl(E)\qquad\text{over }|E|=m, \qquad (\alpha,\beta,m)\in[0,\infty)\times[0,\infty)\times(0,\infty),
	\]
where $\nl(E)$ is defined as in \eqref{eqn: Nonlocal Definition}. Here we are studying $\enf=\calE_{\alpha,1}$ over volume $m$, whereas in the literature authors sometimes consider $\ene$ over volume $|B_1|$ instead. Under a volume scaling, studying volume-constrained minimizers of any of these two functionals is equivalent to just studying $\enf$ for a particular choice of $\alpha$ and $m$. Indeed, by a change of variables 
	\[
		\nl(tE) = \int_{tE}\int_{tE} Y_{\alpha}(x-y) \dd x \dd y = t^6 \int_E\int_E Y_{\alpha}(t(x-y)) \dd x\dd y = t^5\calN_{t\alpha}(E)
	\]
and combined with the usual scaling of the perimeter yields
	\begin{equation}\label{eqn:dilation}
  		\ene(tE)=t^2\bigl[\Pe(E)+\beta t^3\calN_{t\alpha}(E)\bigr];
	\end{equation}
hence, $(\alpha,\beta,m)\mapsto(t\alpha,\,\beta t^3,\,m/t^3)$. For example, minimizers of $\ene$ with volume $m$, provided they exist, are (up to scaling) minimizers of $\calF_{\alpha\beta^{-1/3}}$ with volume $m\beta$. At times we may switch to studying $\ene$ for the convenience of presentation. Indeed, Theorem \ref{thm: Stable Critical Threshold} and Theorem \ref{thm: L1 Local Minimality} are more easily stated and proven in terms of $\ene$; see Theorem \ref{thm: Stable Critical Threshold Beta} and Theorem \ref{thm: L1 Local Minimality Beta} for the analogous statements, including an explicit stability threshold in terms of $\beta$, see \eqref{eqn:beta_star}. We state all the results in the introduction in terms of $\enf$ to not burden the reader and to provide a physically relevant functional. 
\end{remark}	

\begin{remark}[Comparison with previous literature]\label{rem:comparison}
At $\alpha=0$ the value of $m_*(0)$ is not new. Bonacini and Cristoferi \cite{BC} and later Figalli, Fusco, Maggi, Millot, and Morini \cite{FFMMM} obtain the unscreened counterpart of Theorem~\ref{thm: L1 Local Minimality}, i.e. the sharp threshold for local minimality of balls for the perimeter plus a Riesz interaction. In $\R^3$ our $m_*(\alpha)$ agrees with it as $\alpha\to0$.

In \cite{Peg}, Pegon considers the functionals
	\[
		\mathcal{F}_{\gamma, G_{\lambda}}(E) \defeq  \Pe(E)-\gamma\lambda^4 \int_E\int_{E^c} G\big(\lambda (x-y)\big) \dd x \dd y
	\]
for some kernel $G$. In particular, if $G(x) = e^{-|x|}/(2\pi |x|)$ then Pegon proves in \cite[Theorem C]{Peg} that
	\begin{itemize}
		\item if $\gamma<1$ there exists $\lambda_s(\gamma)>0$ such that if $\lambda>\lambda_s$, then $B_1$ is a volume-constrained stable critical point of $\mathcal{F}_{\gamma,G_{\lambda}}$;
		\item if $\gamma >1$  there exists $\lambda_u(\gamma)>0$ such that if $\lambda>\lambda_u$, then $B_1$ is a volume-constrained unstable critical point of $\mathcal{F}_{\gamma,G_{\lambda}}$.
	\end{itemize}	
Using that $\|Y_{\lambda}\|_{L^1(\mathbb{R}^3)} = 4\pi/\lambda^2$, we have
	\[
		\mathcal{F}_{\gamma, G_{\lambda}}(E) = \Pe(E) + \gamma\lambda^3 \int_E\int_E \frac{e^{-\lambda|x-y|}}{2\pi|x-y|}  \dd x \dd y - 2\gamma \lambda |E| = \calE_{\lambda, \gamma \lambda^3/2\pi}(E)-2\gamma\lambda|E|.
		\]
Hence, analyzing volume-constrained stable critical points $E$ with $|E|=|B_1|$ of $\mathcal{F}_{\gamma, G_{\lambda}}$ is equivalent to that of $\calF_{\alpha}$ over $E$ with $|E|=m$, where $\alpha=(2\pi/\gamma)^{1/3}$ and $m=(\gamma\lambda^3/2\pi)|B_1|$. In particular, when $\gamma>1$ we have $\alpha<(2\pi)^{1/3}$, and so by Theorem \ref{thm: Stable Critical Threshold} we can recover Pegon's result for some explicitly computable $\lambda_u$, and moreover show that if $\lambda \leq \lambda_u$ then $B_1$ is a volume-constrained stable critical point of $\calF_{\gamma, G_{\lambda}}$. On the other hand if $\gamma \leq 1$ then $\alpha \geq (2\pi)^{1/3}$, and so $B_1$ is \textit{always} a volume-constrained stable critical point; hence $\lambda_s=0$.
\end{remark}

\begin{remark} One could also study an analogue of $\enf$ in higher dimensions, whose kernel arises as the fundamental solution of $-\Delta+\alpha^2$. For the sake of presentation we elect to focus on dimension $n=3$ as the kernel is well studied in that case, see \cite{Winkelmann}. However, we believe that similar results hold in any dimension.
\end{remark}

 \subsection*{Acknowledgments}
LB is supported by an NSERC discovery grant. KDM was supported by NSF-DMS 2306962. This work was partially completed while KDM was hosted at the Department of Mathematics and Statistics of Virginia Commonwealth University. IT was partially supported by the Simons Foundation MPS-TSM 851065, and by NSF-DMS 2306962 and NSF-DMS 2607148. The authors would also like to acknowledge support from the Fields Institute through their Fields Opportunity for Collaborations US (FOCUS) program.


\section{Notation and Preliminaries}\label{sec:notation_prelim}

As in \cite{Peg}, we will denote by $\calP_\alpha$ the nonlocal functional
	\[
		\calP_\alpha(E) \defeq \int_E \int_{E^c} Y_\alpha(x-y)\dd x \dd y.
	\]
Since $\| Y_\alpha\|_{L^1(\R^3)}=4\pi/\alpha^2$, recalling \eqref{eqn: Nonlocal Definition}, we have $\nl(E)=4\pi|E|/\alpha^2-\calP_\alpha(E)$, so $\enf(E)=\eng(E)+4\pi|E|/\alpha^2$, where
	\[
		\eng(E)\defeq \Pe(E)-\calP_\alpha(E).
	\]
As a consequence,
	\[
		g_\alpha(m) \defeq \inf \Big\{ \eng(E) \colon |E|=m \Big\} = e_\alpha(m) - \frac{4\pi m}{\alpha^2},
	\]
where $e_\alpha(m)$ is defined in \eqref{eqn:energy_alpha_min}.

Throughout the paper we will refer to the following constants related to the surface tension: 
	\[
		\sigma_\alpha\defeq \frac{2\pi}{\alpha^3}, \qquad\qquad \theta_\alpha \defeq 1-\sigma_\alpha,
	\]
as well as the following ones related to the local and nonlocal terms of the energy:
	\begin{equation}\label{eqn:constants}
		C_1 \defeq (36\pi)^{1/3}, \qquad\qquad C_2 \defeq \frac{32\pi^2}{15}\left(\frac{3}{4\pi}\right)^{5/3}, \qquad \text{so that}\quad \bar{m}= \frac{2^{1/3}-1}{1-2^{-2/3}}\frac{C_1}{C_2},
	\end{equation}	
where $\bar{m}$ is defined in the statement of Theorem~\ref{thm:exist_every_alpha}.
Here $C_1$ and $C_2$ are such that $\Pe(E)\geq C_1 |E|^{2/3}$, and $\calN_0(B[m])=C_2m^{5/3}$, where $B[m]$ denotes the ball in $\R^3$ of volume $m$. Also note that $\theta_\alpha>0$ simply means $\alpha>(2\pi)^{1/3}$.

In the following lemma we collect the properties of the Yukawa kernel that we will need in the paper.
	\begin{lemma}[Properties of $Y_\alpha$]\label{lem:kernel_prop}
		For any $\alpha\geq 0$, the function $Y_\alpha$ satisfies the following.
			\begin{enumerate}
            \item $Y_\alpha$ is radial, strictly decreasing in $|z|$, and $0<Y_\alpha \leq Y_0=1/|z|$.
            \item $|\nabla Y_\alpha(z)|=(1+\alpha|z|)e^{-\alpha|z|}/|z|^2\leq 1/|z|^2$.
            \item $Y_\alpha$ is positive definite. Namely $\trns{Y_\alpha}(k)=4\pi/(\alpha^2+|k|^2)>0$, and so
            	\[
            		\iint Y_\alpha(x-y)f(x)f(y)\dd x \dd y \geq 0
            	\]
            for any $f\in L^1(\R^3)\cap L^2(\R^3)$ with compact support.
            \item For any $E\subset \R^3$ with $|E|\leq |B_1|$, and for any $x\in\R^3$, we have
            	\begin{gather*}
            		\int_E Y_\alpha(x-y)\dd y \leq \int_{B_1} \frac{1}{|y|}\dd y = 2\pi,\\
            		\int_E |\nabla Y_\alpha(x-y)|\dd y \leq \int_{B_1} \frac{1}{|y|^2} \dd y = 4\pi.
            	\end{gather*}
            \item For any $E\subset \mathbb{R}^3$ we have $\nl(E) \leq \nl(B[|E|]) \leq \calN_0(B[|E|])$. In particular, $\nl(E) \leq C_2 |E|^{5/3}$ and $\nl(B_1) \leq 32\pi^2/15$. 
        \end{enumerate}
	\end{lemma}
	
\begin{proof}
The first item follows from the definition of $Y_\alpha$. For the second item we use $(1+t)e^{-t}\leq 1$. The third one follows from the fact that $(-\Delta+\alpha^2)Y_\alpha=4\pi\delta_0$. Items (iv) and (v) use the fact that $Y_\alpha \leq Y_0$ combined with the Riesz rearrangement inequality.
\end{proof}

Next we collect some additional properties about the Yukawa energy of $B_1$.
	\begin{lemma}[Properties of $\nl(B_1)$]\label{lem: Properties of Yukawa B_1} For any $\alpha\geq 0$ the following hold:
		\begin{enumerate}
			\item $\nl(B_1)$ is explicitly given by
			\[
				\nl(B_1) = \frac{8\pi^2e^{-\alpha}}{3\alpha^5}\Big(\left(2\alpha^3-3\alpha^2+3\right) e^{\alpha}-3(\alpha+1)^2e^{-\alpha}\Big).
			\]
			\item Denoting $\nl^{(k)}(B_1)=\dd^k/\dd \alpha^k\, \nl(B_1)$, there exists a $t_0>0$, independent of $\alpha$, such that if $t\in (-t_0,t_0)$ then
			\[
				\nl(B_1)+\alpha\nl'(B_1)\,t+\alpha^3\nl'''(B_1)\,\frac{t^3}{6} \geq 0.
			\]
			\item The following inequalities hold
			\begin{align*}
				&\nl''(B_1)>0, \qquad\qquad 50\nl(B_1)+58\alpha \nl'(B_1)+15\alpha^2\nl''(B_1)+\alpha^3\nl'''(B_1)>0, \\
				&\qquad \quad  60\nl'(B_1)+30\alpha\nl''(B_1)+5\alpha^2\nl'''(B_1) <0, \qquad \qquad \nl'''(B_1)<0.
			\end{align*}
		\end{enumerate}
	\end{lemma}
	
\begin{proof}
Without loss of generality assume that $y=(0,0,\rho)$. Then in spherical coordinates
	\begin{align*}
		\int_{B_1} \frac{e^{-\alpha |x-y|}}{|x-y|} \dd x &= 2\pi \int_0^1 \int_0^{\pi} \frac{e^{-\alpha \sqrt{\rho^2+r^2-2r\rho\cos(\varphi)}}}{\sqrt{\rho^2+r^2-2r\rho\cos(\varphi)}} \, r^2\sin(\varphi) \dd\varphi \dd r \\
											&=\frac{2\pi}{\rho}\int_0^1 r\int_{|\rho-r|}^{\rho+r} e^{-\alpha u} \dd u \dd r=\frac{2\pi}{\alpha \rho}\int_0^1 r(e^{-\alpha|\rho-r|}-e^{-\alpha(\rho+r)}) \dd r =f(\rho),
\end{align*}
where
	\[
		f(\rho) \defeq \frac{2\pi}{\alpha^3 \rho}\Big((\alpha+1)e^{-\alpha(1+\rho)}-(\alpha+1)e^{-\alpha(1-\rho)}+2\alpha\rho \Big).
	\]
Consequently, we have that $\nl(B_1)$ is given by
	\begin{align*}
		\nl(B_1)&=\int_{B_1}\int_{B_1} \frac{e^{-\alpha|x-y|}}{|x-y|} \dd x \dd y = 4\pi\int_0^1 \rho^2f(\rho) \dd\rho \nonumber \\
&=\frac{8\pi^2e^{-\alpha}}{3\alpha^5}\Big(\left(2\alpha^3-3\alpha^2+3\right) e^{\alpha}-3(\alpha+1)^2e^{-\alpha}\Big).
	\end{align*}
Items (ii) and (iii) then follow by direct computation.
\end{proof}

Defining the potential of the unit ball as
	\[
		v_\alpha(x) \defeq \int_{B_1} Y_\alpha(x-y)\dd y
	\]
we obtain the following lemma that will be used in the sequel as well. The radially symmetric decreasing property of $v_\alpha$ is a simple consequence of the layer cake representation (see for example \cite[Theorem 1.13]{LL}) and the bound on the Lipschitz constant follows from an explicit calculation.
	\begin{lemma}\label{lem:potential}
		For any $\alpha\geq 0$, the potential $v_\alpha$ is radial and radially decreasing. Moreover $\Lip(v_\alpha) \leq 4\pi$.
	\end{lemma}

\section{Existence of minimizers of \texorpdfstring{$\enf$}{F\_α}}\label{sec:exist_min_Falpha}

In this section we prove Theorems \ref{thm:exist_every_m} and \ref{thm:exist_every_alpha} following \cite{MNS}. The main tool in their approach is the boundary representation of the energy $\enf$, which isolates the effective tension and a nonnegative defect. For this, we define
	\[
		\psi_\alpha(r) \defeq \frac{(1+\alpha r)e^{-\alpha r}}{\alpha^2 r^2},
	\]
and for any $y\in\pt^* E$, we let
	\[
		H_y \defeq \big\{x \colon \nu_E(y)\cdot (x-y)<0\big\}
	\]
be the inner-tangent half space, where $\nu_E$ denotes the outer unit normal to $E$.

	\begin{proposition}[Boundary representation, cf. {\cite[Lemma 3.1]{MNS}}]\label{prop:bd_repres}
		For any $\alpha>0$, we have
			\[
				\Pe_\alpha(E) = \int_{\pt^* E} \int_{E} \psi_{\alpha}(|x-y|)\nu_E(y)\cdot \frac{y-x}{|y-x|} \dd x \dd \Htwo(y) = \sigma_\alpha \Pe(E) - \calD_\alpha(E),
			\]
		where
			\[
				\calD_\alpha(E) \defeq \int_{\pt^* E} \int_{\asymm{E}{H_y}} \psi_\alpha(|x-y|)\left|\nu_E(y)\cdot \frac{y-x}{|y-x|}\right| \dd x \dd\Htwo(y) \geq 0.
			\]
		Consequently, $\Pe_\alpha(E) \leq \sigma_\alpha \Pe(E)$ and $\enf(E) = \theta_\alpha \Pe(E) + \calD_\alpha(E) + 4\pi|E|/\alpha^2$.
	\end{proposition}
	
\begin{proof}
Let $\Phi_\alpha=Y_\alpha/\alpha^2=e^{-\alpha|z|}/(\alpha^2|z|)$. Then $\nabla \Phi_\alpha(z)=-\psi_\alpha(|z|)\,z/|z|$, and
	\[
		\Delta \Phi_\alpha = \frac{\Delta Y_\alpha}{\alpha^2}=Y_\alpha \qquad \text{on } \R^3 \setminus \{0\}.
	\]
	
Now fix a point $x\in E$ of density one, and consider the function $y\mapsto \nabla_y \Phi_\alpha(x-y)$. Choose $\theta,\eta \in C^\infty(\R)$ such that
	\begin{gather*}
		\theta=0 \text{ on }(-\infty,1], \quad \theta=1 \text{ on } [2,\infty), \quad 0\leq \theta \leq 1,\\
		\eta=1 \text{ on }(-\infty,1], \quad \eta=0 \text{ on } [2,\infty), \quad 0\leq \eta \leq 1.
	\end{gather*}
For $0<\e<R$, define $\Xi_{\e,R}(y) \defeq \theta\left(|x-y|/\e\right)\eta\left(|x-y|/R\right)$. Then $\Xi_{\e,R}\equiv 0$ for $|x-y|\leq \e$ and for $|x-y|\geq 2R$; hence, the singularities at $y=x$ and $y=\infty$ are removed, and
	\[
		V_{\e,R}(y) \defeq \Xi_{\e,R}(y) \nabla_y \Phi_\alpha(x-y) \in C_c^1(\R^3).
	\]
By the divergence theorem, applied in $E^c$, we get
	\begin{align*}
		-\int_{\pt^* E} \Xi_{\e,R} \nabla_y\Phi_\alpha(x-y)\cdot \nu_E(y) \dd \Htwo(y) &= \int_{\pt^* E^c} V_{\e,R}(y) \cdot \nu_{E^c}(y) \dd \Htwo(y) \\
					&= \int_{E^c} \dive_y V_{\e,R}(y)\dd y =I_1+I_2,
	\end{align*}
where
	\[
		I_1:=\int_{E^c} \Xi_{\e,R} Y_\alpha(x-y)\dd y \qquad \text{and} \qquad I_2:= \int_{E^c} \nabla_y \Xi_{\e,R}(y) \cdot \nabla_y \Phi_\alpha(x-y)\dd y.
	\]
The dominated convergence theorem implies that
	\[
		I_1 \to \int_{E^c} Y_\alpha(x-y)\dd y \qquad \text{as }\e\to 0 \text{ and } R\to\infty.
	\]
In order to estimate $I_2$ we will write it as a sum of an inner shell and an outer shell. Precisely,
	\begin{align*}
		I_2 &= I_{2,\mathrm{inn}} + I_{2,\mathrm{out}} \\
			&\defeq \int_{E^c} \eta\left(\frac{|x-y|}{R}\right)\nabla_y\left[\theta\left(\frac{|x-y|}{\e}\right)\right] \cdot \nabla_y \Phi_\alpha(x-y)\dd y \\
			&\qquad\qquad + \int_{E^c} \theta\left(\frac{|x-y|}{\e}\right)\nabla_y\left[\eta\left(\frac{|x-y|}{R}\right)\right] \cdot \nabla_y \Phi_\alpha(x-y)\dd y.
	\end{align*}
The outer shell is supported on $\{R\leq |x-y|\leq 2R\}$ with $|\nabla_y \eta| \leq C/R$. On the support,
	\[
		|\nabla_y \Phi_\alpha(x-y)| = \psi_\alpha(|x-y|) \leq \frac{(1+2\alpha R)e^{-\alpha R}}{\alpha^2 R^2} \leq \frac{C e^{-\alpha R}}{\alpha R},
	\]
and the shell has volume $O(R^3)$. Therefore
	\[
		|I_{2,\mathrm{out}}| \leq \frac{C}{R} \frac{e^{-\alpha R}}{\alpha R} R^3 \to 0 \qquad \text{ as } R\to \infty.
	\]
The inner shell, on the other hand, is supported on $\{\e \leq |x-y| \leq 2\e\}$ with $|\nabla_y \theta| \leq C/\e$. On the support,
	\[
		|\nabla_y \Phi_\alpha(x-y)| = \psi_\alpha(|x-y|) \leq \frac{1}{\alpha^2 \e^2}.
	\]
Thus
	\[
		|I_{2,\mathrm{inn}}| \leq \frac{C}{\e} \frac{1}{\alpha^2\e^2} |E^c \cap B_{2\e}(x)| \to 0 \qquad \text{ as } \e \to 0,
	\]
since $|E^c \cap B_{2\e}(x)|=o(\e^3)$ due to the fact that $x$ is a point of density one. Putting these together yields $I_2 \to 0$ as $\e\to 0$ and $R\to \infty$.

Next, we will consider the boundary integral. First note that for almost every $x\in E$, $\int_{\pt^* E} \psi_\alpha(|x-y|)\dd \Htwo(y)$ is finite, since, by symmetric rearrangement,
	\[
		\int_E \int_{\pt^* E} \psi_\alpha(|x-y|)\dd \Htwo(y) \dd x = \int_{\pt^* E} \int_E \psi_\alpha(|x-y|)\dd x \dd \Htwo(y) \leq \frac{4\pi\rho}{\alpha^2}\Pe(E),
	\]
where $\rho$ is such that $|B_\rho|=|E|$. Then, by the dominated convergence theorem,
	\[
		\int_{\pt^* E} \Xi_{\e,R} \nabla_y\Phi_\alpha(x-y)\cdot \nu_E(y) \dd \Htwo(y) \to \int_{\pt^* E}  \nabla_y\Phi_\alpha(x-y)\cdot \nu_E(y) \dd \Htwo(y),
	\]
as $\e\to 0$ and $R\to \infty$. Therefore,
	\[
		\int_{E^c} Y_\alpha(x-y)\dd y = -\int_{\pt^* E}\nabla_y\Phi_\alpha(x-y)\cdot \nu_E(y) \dd \Htwo(y)=\int_{\pt^* E} \psi_\alpha(|x-y|)\,\nu_E(y)\cdot\frac{y-x}{|y-x|}\dd \Htwo(y),
	\]
and integrating over $E$ and using Fubini's theorem yields
	\[
		\Pe_\alpha(E)=\int_{\pt^* E}\int_E \psi_\alpha(|x-y|)\,\nu_E(y) \cdot \frac{y-x}{|y-x|}\dd x \dd \Htwo(y).
	\]
Fixing $y\in\pt^* E$, and letting $w=x-y$, we have that $x\in H_y$ if and only if $\nu_E \cdot w <0$. Hence,
	\begin{align*}
		\int_{H_y} \psi_\alpha(|x-y|)\nu_E(y)\cdot \frac{y-x}{|y-x|}\dd x &= \int_{\{\nu_E\cdot w<0\}} \psi_\alpha(|w|) \left|\nu_E \cdot \frac{w}{|w|}\right|\dd w  \\
									&= \int_0^\infty \psi_\alpha(r) r^2 \dd r \, \int_{\{\nu_E\cdot\omega<0\}} |\nu_E\cdot\omega|\dd \Htwo(\omega) = \frac{2\pi}{\alpha^3} = \sigma_\alpha,
	\end{align*}
which is independent of $y$. Now, note that on $E\,\setminus H_y$ we have $\nu_E\cdot(y-x)\leq 0$, so 	
	\[
		\psi_\alpha(|x-y|)\nu_E(y)\cdot \frac{y-x}{|y-x|}\leq 0,
	\]
and reverse inequalities hold on $H_y\setminus E$. Thus, writing $E=H_y\cup (H_y\setminus E) \cup (E\,\setminus H_y)$, we obtain
	\[
		\int_{E} \psi_\alpha(|x-y|)\nu_E(y)\cdot\frac{y-x}{|y-x|}\dd x =\sigma_\alpha - \int_{\asymm{E}{H_y}} \psi_\alpha(|x-y|)\left|\nu_E(y)\cdot\frac{y-x}{|y-x|}\right|\dd x.
	\]
Finally, integrating over $\pt^* E$ we get $\Pe_\alpha(E)=\sigma_\alpha \Pe(E)-\calD_\alpha(E)$.
\end{proof}

The next lemma gives the scaling and monotonicity properties of $\calD_\alpha$. It follows easily from the scaling properties of the perimeter, the kernel $Y_\alpha$, and the functional $\Pe_\alpha$ combined with the monotonicity of the function $t\mapsto (1+t)e^{-t}$.

	\begin{lemma}\label{lem:D_alpha_properties}
		For any $t>0$, $\calD_\alpha(tE) = t^5 \calD_{t\alpha}(E)$. Moreover, for any $0<\lambda\leq 1$ and $r>0$, $\lambda^2 \psi_{\lambda\alpha}(r) \geq \psi_\alpha(r)$, and $\lambda^2 \calD_{\lambda\alpha}(E) \geq \calD_\alpha(E)$ for any $E\subset \R^3$.
	\end{lemma}

Finally, the core of the proof of Theorem~\ref{thm:exist_every_m} relies on the following compactness result which is a classical consequence of strict subadditivity, and can be obtained by arguing as in \cite[Proposition 2.1, Lemma 2.2, Theorem 3.1]{FL} or \cite[Proposition 1.2, Lemma 3.4]{NP} combined with the continuity of $\nl$.

	\begin{proposition}\label{prop:compactness}
		For any $\alpha\geq 0$ and $m>0$, if
			\[
				e_\alpha(m) < e_\alpha(\xi) + e_\alpha(m-\xi) \qquad \text{ for all }\xi\in(0,m),
			\]
		then $e_\alpha(m)$ is attained.
	\end{proposition}
We are now ready to prove Theorem \ref{thm:exist_every_m}. 
  
\begin{proof}[Proof of Theorem~\ref{thm:exist_every_m}] In light of Proposition \ref{prop:bd_repres} we have $\eng=\theta_\alpha \Pe + \calD_\alpha$. By Lemma~\ref{lem:D_alpha_properties}, we have that
	\[
		\calD_\alpha(t^{1/3}E) = t^{5/3}\calD_{t^{1/3}\alpha}(E) \geq t\calD_\alpha(E).
	\]
Adding $\theta_\alpha t^{2/3}\Pe(E)$ to both sides then yields
	\begin{align*}
		\eng(t^{1/3}E) &= \theta_\alpha \Pe(t^{1/3}E) + \calD_\alpha(t^{1/3}E) = \theta_\alpha t^{2/3}\Pe(E) + \calD_\alpha(t^{1/3}E) \\
					   &\geq \theta_\alpha(t+\varphi(t))\Pe(E) + t\calD_\alpha(E) = t\eng(E)+\theta_\alpha\varphi(t)\Pe(E),
	\end{align*}
where $\varphi(t)\defeq t^{2/3}-t>0$ on $(0,1)$. Now, letting $t=m'/m$ for $0<m'<m$ and taking the infimum of both sides over sets of volume $m$ we obtain
	\begin{equation}\label{eqn:g_alpha_ineq}
		g_\alpha(m') \geq \frac{m'}{m}g_\alpha(m) + \theta_\alpha C_1 m^{2/3}\varphi\left(\frac{m'}{m}\right),
	\end{equation}
where we also used the inequality $\Pe(E) \geq C_1 m^{2/3}$. Adding $4\pi m'/\alpha^2$ to both sides, and using the fact that $\varphi(m'/m)>0$ we get that
	\[
		\frac{e_\alpha(m')}{m'} > \frac{e_\alpha(m)}{m};
	\]
hence, $m^{-1}e_\alpha(m)$ is strictly decreasing.

For $0<\xi<m$, we can apply the inequality \eqref{eqn:g_alpha_ineq} with $m'=\xi$ and $m'=m-\xi$. Adding these two inequalities we obtain
	\[
		e_\alpha(\xi) + e_\alpha(m-\xi) \geq e_\alpha(m) + \theta_\alpha C_1 m^{2/3}\left[\varphi\left(\frac{\xi}{m}\right)+\varphi\left(\frac{m-\xi}{m}\right)\right];
	\]
hence, again thanks to the positivity of $\varphi$, we get the strict binding inequality
	\[	
		e_\alpha(m) < e_\alpha(\xi) + e_\alpha(m-\xi),
	\]
and the result follows from Proposition~\ref{prop:compactness}.
\end{proof}

\begin{remark} The only estimates used to prove the strict binding inequality above come from Lemma~\ref{lem:D_alpha_properties} and the isoperimetric inequality. By quantifying these, one could potentially prove a stronger strict binding inequality. For instance, we have that 
	\[
		\lambda^2\psi_{\lambda\alpha}(r)-\psi_{\alpha}(r) = \frac{1}{\alpha^2r^2}\int_{\lambda\alpha r}^{\alpha r} te^{-t} \dd t \geq \frac{1-\lambda^2}{2}e^{-\alpha r},
	\]
which leads to the corresponding estimate for $\calD_{\alpha}$ as
	\[
		\lambda^2\calD_{\lambda \alpha}(E)\geq \calD_{\alpha}(E)+\frac{1-\lambda^2}{2}\int_{\partial^*E}\int_{E\Delta H_y} e^{-\alpha |x-y|} \left|\nu_E(y)\cdot \frac{x-y}{|x-y|}\right| \dd x\dd \mathcal{H}^2(y).
	\]
If we na\"{i}vely use the argument of Proposition \ref{prop:bd_repres} in reverse to rewrite the latter integral, then 
	\[
		\int_{\partial^*E}\int_{E\Delta H_y} e^{-\alpha |x-y|} \left|\nu_E(y)\cdot \frac{x-y}{|x-y|}\right| \dd x\dd \mathcal{H}^2(y) = \sigma_{\alpha}\Pe(E)-\frac{1}{\alpha}\frac{\dd}{\dd \alpha}\left(\alpha^2\nl(E)\right);
	\]	
however, it is unclear if this gives a better result. 
\end{remark}

Now we turn to the proof of Theorem~\ref{thm:exist_every_alpha}. For this we will follow the arguments in \cite{FN} and utilize the following lemma, whose proof requires a minor modification of \cite[Lemma 5]{FN}.

	\begin{lemma}
		For any $\alpha\geq 0$ and $0<m'<m$, if $t=m'/m$, then
			\begin{equation}\label{eqn:min_at_m'}
				e_\alpha(m') \geq t^{5/3}e_\alpha(m) + (1-t)t^{2/3}C_1 m^{2/3}.
			\end{equation}
	\end{lemma}
	
\begin{proof}
Let $E\subset\R^3$ be a set such that $|E|=m'$. Then $|t^{-1/3}E|=m$, and $\Pe(t^{-1/3}E)=t^{-2/3}\Pe(E)$, and
	\[
		\nl(t^{-1/3}E)=t^{-5/3}\calN_{t^{-1/3}\alpha}(E) \leq t^{-5/3}\nl(E),
	\]
since $t^{-1/3}\alpha\geq \alpha$ (for $\alpha=0$ we have an equality above and recover the homogeneous case of \cite{FN}). Therefore
	\[
		e_\alpha(m) \leq \enf(t^{-1/3}E) = t^{-5/3}\enf(E) - \big(t^{-5/3}-t^{-2/3}\big)\Pe(E).
	\]
Since $t<1$ then $t^{-5/3}-t^{-2/3}>0$, and by the isoperimetric inequality, $\Pe(E)\geq C_1 t^{2/3}m^{2/3}$; hence,
	\[
		e_\alpha(m) \leq t^{-5/3}\enf(E) - \big(t^{-5/3}-t^{-2/3}\big)t^{2/3} C_1 m^{2/3}.
	\]
Taking the infimum over sets of volume $m'$ and multiplying both sides by $t^{5/3}$ yields the inequality \eqref{eqn:min_at_m'}.
\end{proof}

We now prove Theorem \ref{thm:exist_every_alpha} following \cite[Theorem 4]{FN}.

\begin{proof}[Proof of Theorem~\ref{thm:exist_every_alpha}]
We apply the previous lemma twice, with $m'=\xi$ and $m'=m-\xi$. Then adding both inequalities and letting $t=\xi/m$, we obtain
	\begin{multline*}
		e_\alpha(\xi)+e_\alpha(m-\xi)-e_\alpha(m) \\ \geq \big(t^{5/3}+(1-t)^{5/3}-1\big)e_\alpha(m) + \big((1-t)t^{2/3}+t(1-t)^{2/3}\big)C_1 m^{2/3}.
	\end{multline*}
Since $t^{5/3}+(1-t)^{5/3}-1<0$, by using the ball of volume $m$ as a competitor,
	\begin{align*}
		e_\alpha(\xi)+e_\alpha(m-\xi)&-e_\alpha(m) \geq C_1 m^{2/3}\Big[\big(t^{5/3}+(1-t)^{5/3}-1\big)+\big((1-t)t^{2/3}+t(1-t)^{2/3}\big)\Big] \\
						&\qquad\qquad\qquad\qquad\qquad\qquad\qquad\qquad\quad + \nl(B[m])\big(t^{5/3}+(1-t)^{5/3}-1\big) \\
						&\qquad =C_1\big(t^{2/3}+(1-t)^{2/3}-1\big)m^{2/3}+\big(t^{5/3}+(1-t)^{5/3}-1\big)\nl(B[m]).
	\end{align*}
Since $\nl(B[m])\leq \calN_0(B[m])=C_2m^{5/3}$ and, by \cite[Eq. (9)]{FN},
	\[
		\inf_{t\in(0,1)} \frac{t^{2/3}+(1-t)^{2/3}-1}{1-t^{5/3}-(1-t)^{5/3}}=\frac{2^{1/3}-1}{1-2^{-2/3}},
	\]
we get that the right-hand side of the inequality above is positive for $m<\bar{m}$ with $\bar{m} $ defined in \eqref{eqn:constants}. Thus we have the strict binding inequality
	\[
		e_\alpha(m) < e_\alpha(\xi)+e_\alpha(m-\xi),
	\]
which, in turn, implies the theorem via Proposition~\ref{prop:compactness}.
\end{proof}

\begin{remark}[Optimality of $\bar{m}$]\label{rem:optimal}
The threshold $\bar m$ is exactly the volume at which a ball has the same energy for $\alpha=0$ as two balls of half the volume placed infinitely far apart. This is a ceiling for any strict-binding argument: beyond it, if balls are minimizers then $e_0(m)>2e_0(m/2)$ and strict subadditivity is false. In \cite{FN}, where the nonlocal term has a factor of $1/2$, the authors obtain the same threshold $m_*=2\bar m$.
\end{remark}

\begin{remark}[An $\alpha$-dependent refinement]
The kernel enters Theorem~\ref{thm:exist_every_alpha} only through the ball competitor, and with its closed form $\nl(B_R)=16\pi^2/\alpha^5\,\Psi(\alpha R)$, where $\Psi(u)=u^3/3-(1+u)(u\cosh u-\sinh u)e^{-u}$, it yields thresholds $m(\alpha)$ strictly larger than $\bar{m}$ for every $\alpha>0$. However $\nl(B[m])$ decreases in $\alpha$, so the infimum over $\alpha$ is achieved when $\alpha=0$ and a threshold uniform in $\alpha$ cannot be improved this way. Moreover $m(\alpha)$ is increasing, whereas for $\sigma_\alpha<1$ existence already holds at every volume by Theorem~\ref{thm:exist_every_m}. 
\end{remark}


\section{Minimality of balls for \texorpdfstring{$\enf$}{F\_α}}\label{sec:min_balls}	
In this section we prove Theorem~\ref{thm:global_min}.  We first observe owing to Remark \ref{rem:master} that we can instead consider the functional
	\[
		\ene(E) = \Pe(E)+\beta \calN_{\alpha}(E)
	\]
for $E\subset \mathbb{R}^3$ with $|E|=|B_1|$. With this reformulation, Theorem~\ref{thm:global_min} follows from
\begin{theorem}\label{thm:global_min_beta}
		There exists $\beta_0>0$ (independent of $\alpha$) such that for any $\alpha\geq 0$ and $\beta<\beta_0$, the unique minimizer, up to translations, of $\ene$ over sets $|E|=|B_1|$ is the unit ball $B_1$. 
	\end{theorem}
	
\begin{proof}[Proof of Theorem~\ref{thm:global_min}] For $m>0$ and $\alpha\geq 0$, by scaling, see \eqref{eqn:dilation}, we have that $B[m]$ minimizes $\enf(E)$ among sets $|E|=m$ if and only if $B_1$ minimizes $\calE_{\alpha',\beta}(E)$ among sets $|E|=|B_1|$, where $\alpha'=\alpha\beta^{1/3}$ with $\beta = m/|B_1|$. Hence taking $m_0=\beta_0|B_1|$ we have $\beta<\beta_0$ whenever $m<m_0$, and Theorem~\ref{thm:global_min_beta} gives the result.
\end{proof}

We prove Theorem~\ref{thm:global_min_beta} following the arguments in \cite{CFP}. The main tool in establishing the theorem is the following lemma which improves the deficit of the nonlocal term when the set $E$ is nearly spherical. While in \cite{CFP} the authors prove a similar result by explicitly controlling the interactions between $E\,\setminus B_1$ and $B_1\setminus E$, here the positive definiteness of $Y_\alpha$ simplifies the proof.

	\begin{lemma}\label{lem:nonloc_L2_control}
		Let $\alpha\geq 0$ and suppose $E$ is a nearly spherical set defined as
			\[
				E = \big\{ tx\in \R^3 \colon x\in\pt B_1, \ 0\leq t<1+u(x) \big\}
			\]
		for some $u\in C^1(\pt B_1)$ such that $|E|=|B_1|$ and $\|u\|_{L^{\infty}(\pt B_1)}\leq 1/2$. Then
			\[
				\nl(B_1)-\nl(E) \leq 9\pi\|u\|_{L^2(\pt B_1)}^2.
			\]
	\end{lemma}
	
\begin{proof}
Note that the function $\chi_E-\chi_{B_1} \in L^1(\R^3)\cap L^\infty(\R^3)$ and has compact support. Hence, by Lemma~\ref{lem:kernel_prop}(iii), 
	\[
		\int_{\R^3}\int_{\R^3} Y_\alpha(x-y)\big(\chi_E-\chi_{B_1}\big)(x)\big(\chi_E-\chi_{B_1}\big)(y)\dd x \dd y \geq 0.
	\]
Expanding the energy of $E$, this implies that
	\[
		\nl(B_1)-\nl(E) \leq -2\int_{\R^3}\big(\chi_E-\chi_{B_1}\big)(x)v_\alpha(x)\dd x.
	\]
Now, let $\bar{v}_\alpha\defeq v_\alpha(x_0)$ for any $x_0\in\pt B_1$, which is a constant due to Lemma~\ref{lem:potential}. Since $\int_{\R^3}(\chi_E-\chi_{B_1})(x)\dd x =0$ and $|\chi_E-\chi_{B_1}|\leq 1$ with $\chi_E-\chi_{B_1}$ supported on $\asymm{E}{B_1}$, by triangle inequality,
	\begin{multline*}
			-2\int_{\R^3} \big(\chi_E-\chi_{B_1}\big)(x)v_\alpha(x)\dd x \\= -2\int_{\R^3} \big(\chi_E-\chi_{B_1}\big)(x)(v_\alpha(x)-\bar{v}_\alpha)\dd x \leq 2 \int_{\asymm{E}{B_1}}|v_\alpha(x)-\bar{v}_{\alpha}|\dd x. 
	\end{multline*}
This, again using Lemma~\ref{lem:potential}, implies that
	\[
		\nl(B_1)-\nl(E) \leq 8\pi\int_{\asymm{E}{B_1}}\big||x|-1\big|\dd x \leq 8\pi\cdot \frac{9}{4}\int_{\pt B_1} \frac{u^2(\theta)}{2}\dd \theta = 9\pi \|u\|_{L^2(\pt B_1)}^2,
	\]
where the second inequality uses the fact that $(1+\|u\|_{L^\infty(\partial B_1)})^2\leq 9/4$.
\end{proof}

As a consequence we get the following regularity result, which follows from \cite[Chapter 21]{Mag} combined with arguments in \cite{CFP,GMP}.

	\begin{proposition}\label{prop:regularity}
		There exists $\beta_1>0$, $\gamma\in(0,1/2)$, and an increasing function $\eta\colon(0,\beta_1)\to(0,\infty)$ with $\eta(\beta)\to 0$ as $\beta\to 0$, all independent of $\alpha$, such that if $\alpha\geq 0$, $0<\beta<\beta_1$, and $E$ minimizes $\ene$ among sets of volume $|B_1|$, then, up to a translation, $E$ is a nearly spherical set with barycenter at the origin and $\|u\|_{C^{1,\gamma}(\pt B_1)}\leq \eta(\beta)$.
	\end{proposition}

\begin{proof}[Proof of Theorem~\ref{thm:global_min}]
By \cite{Fug}, there exist $\delta>0$ and $C_F>0$ such that if $E$ is a nearly spherical set with barycenter at the origin and $\|u\|_{C^1(\pt B_1)}\leq \delta$, then
	\[
		\Pe(E) - \Pe(B_1) \geq C_F \|u\|_{H^1(\pt B_1)}^2.
	\]
Now, take
	\[
		\beta_0 \defeq \min \Big\{\beta_1,\frac{C_F}{9\pi},\sup\big\{\beta<\beta_1 \colon \eta(\beta)\leq \min(\delta,1/2)\big\}\Big\}
	\]
with $\beta_1$ and $\eta$ from the proposition above. Let $\beta<\beta_0$ and let $E$ be a minimizer of $\ene$. Then $E$ is a nearly spherical set with barycenter at the origin and $\|u\|_{C^1(\pt B_1)}\leq \eta(\beta)\leq \min\{\delta,1/2\}$. Using Lemma~\ref{lem:nonloc_L2_control},
	\[
		\ene(E)-\ene(B_1) = \big(\Pe(E)-\Pe(B_1)\big)-\beta\big(\nl(B_1)-\nl(E)\big)\geq \big(C_F-9\pi\beta\big)\|u\|_{H^1(\pt B_1)}^2.
	\]
Since $\beta<C_F/(9\pi)$ the right-hand side is positive unless $u\equiv 0$, i.e. $E=B_1$.
\end{proof}


\section{Sharp stability of \texorpdfstring{$B_1$}{B1}}\label{sec:stability}
In this section we prove Theorem \ref{thm: Stable Critical Threshold}. As in Section \ref{sec:min_balls}, due to Remark \ref{rem:master} it suffices to prove
	\begin{theorem}\label{thm: Stable Critical Threshold Beta} For any $\alpha\geq 0$ the unit ball $B_1$ is a volume-constrained stable critical point of $\ene$ if and only if $\beta \leq \beta_*(\alpha)$ where
		\begin{equation}\label{eqn:beta_star}
			\beta_*(\alpha) =\frac{\alpha^5 e^{\alpha}}{(2\alpha^2-9)\pi e^{\alpha} + \left((\alpha^2+3\alpha+3)^2+\alpha^2(\alpha+1)^2\right)\pi e^{-\alpha}}.
		\end{equation}
	\end{theorem}
	
\begin{proof}[Proof of Theorem~\ref{thm: Stable Critical Threshold}] For $m>0$ and $\alpha \geq 0$, by scaling we have that $B[m]$ is a volume-constrained stable critical point of $\enf$ if and only if $B_1$ is a volume-constrained stable critical point of $\calE_{\alpha', \beta}$ with $\alpha'=\alpha \beta^{1/3}$ and $\beta=m/|B_1|$. Theorem~\ref{thm: Stable Critical Threshold Beta} then gives that $B[m]$ is a volume-constrained stable critical point of $\enf$ if and only if $\beta \leq \beta_*(\alpha')$, i.e.
	\[
		m\leq |B_1|\beta_*\left(\frac{\alpha m^{1/3}}{|B_1|^{1/3}}\right).
	\]
Hence, $m_*(\alpha)$ is given implicitly by
	\begin{equation}\label{eqn:m_star}
		m_*(\alpha) = |B_1|\beta_*\left(\frac{\alpha m_*(\alpha)^{1/3}}{|B_1|^{1/3}}\right).
	\end{equation}
Indeed, an explicit computation shows that $\beta_*(r)/r^3$ is decreasing, and the claim follows by taking $r=\alpha m^{1/3}/|B_1|^{1/3}$ and $r_*=\alpha m_*(\alpha)^{1/3}/|B_1|^{1/3}$. 

Notice that although $\beta_*(\alpha)$ is explicit, we cannot explicitly find $m_*(\alpha)$. However, we can deduce some information about it. We now show that $m_*(\alpha)=+\infty$ for $\alpha \geq (2\pi)^{1/3}$. Rearranging \eqref{eqn:m_star} gives $\alpha = t/\beta_*(t)^{1/3}$ for $t=\alpha (m_*(\alpha)/|B_1|)^{1/3}$. Since $t/\beta_*(t)^{1/3}$ is increasing, it is injective and has a maximum value of $(2\pi)^{1/3}$ at $+\infty$. Hence $m_*(\alpha)$ is uniquely defined for $\alpha<(2\pi)^{1/3}$, and numerically computable via \eqref{eqn:m_star}. Thus, if $\alpha \geq (2\pi)^{1/3}$ then $m_*(\alpha)$ is $+\infty$. Furthermore by implicit differentiation we see that
	\[
	m_*'(\alpha) = \frac{|B_1|\beta_*'(t)\beta_*(t)}{\beta_*(t)^{2/3}-1/3\,t\beta_*'(t)\beta_*(t)^{-1/3}}>0,
	\]
and so $m_*(\alpha)\geq m_*(0)=5$. 
\end{proof}

Let us now turn to proving Theorem~\ref{thm: Stable Critical Threshold Beta}. We first expand the Yukawa potential in spherical harmonics. To this end let $\{Y_l^m\}$ be an orthonormal basis of $\partial B_1$ of spherical harmonics. Recall that if $x_1,x_2 \in \partial B_1$ then
	\begin{equation}\label{eqn: Yukawa Spherical Harmonic Expansion}
		\frac{e^{-\alpha|x_1-x_2|}}{|x_1-x_2|} = 4\pi \alpha \sum_{l=0}^{\infty} i_l(\alpha)k_l(\alpha) \sum_{m=-l}^l Y_l^m(\theta_1,\varphi_1)Y_l^{m^*}(\theta_2,\varphi_2)
	\end{equation}
where $i_l$ and $k_l$ are the modified spherical Bessel functions given by
\[i_l(\alpha) = \sqrt{\frac{\pi}{2\alpha}}\,I_{l+1/2}(\alpha) \quad \text{and} \quad k_l(\alpha) = \sqrt{\frac{2}{\pi \alpha}}\,K_{l+1/2}(\alpha)\]
and $I$ and $K$ are the modified Bessel functions of the first and second kind respectively (cf. \cite[Eq. (39)]{Winkelmann}). Let us define the operator
	\[
		\mathscr{Y}_{\alpha}u(x)\defeq 2\int_{\partial B_1}(u(x)-u(y))Y_{\alpha}(x-y) \dd \mathcal{H}^2(y)
	\]
so that for every $u\in C^1(\partial B_1)$ we have
	\begin{equation}\label{eqn: alpha Norm to Y Operator}
		[u]_{\alpha}^2 := \int_{\partial B_1}\int_{\partial B_1} |u(x)-u(y)|^2Y_{\alpha}(x-y) \ \dd \mathcal{H}^2(y)\dd\mathcal{H}^2(x) = \int_{\partial B_1} u(x)\,\mathscr{Y}_{\alpha}u(x) \dd\mathcal{H}^2(x).
	\end{equation}
We aim to study the eigenvalues of $\mathscr{Y}_{\alpha}$. To this end we introduce the \textit{modified Helmholtz operator on the sphere} as
	\[
		\mathcal{Y}_{\alpha}u(x)\defeq\int_{\partial B_1}u(y)Y_{\alpha}(x-y) \dd\mathcal{H}^2(y)
	\]
and observe that
	\[
		\mathscr{Y}_{\alpha} = 2(\mu_0^*(\alpha)\Id - \mathcal{Y}_{\alpha})\]
where $\mu_0^*(\alpha)$ is the 0-th eigenvalue of $\mathcal{Y}_{\alpha}$ given by
	\[
		\mu_0^*(\alpha) = \int_{\partial B_1}\frac{e^{-\alpha|x-y|}}{|x-y|} \dd\mathcal{H}^2(y).
	\]
It suffices then to determine the eigenvalues $\mu_l^*(\alpha)$ of $\mathcal{Y}_{\alpha}$, that is for any $Y_l^m$ with $m\in \{-l,\ldots,l\}$ we have
	\[
		\mathcal{Y}_{\alpha} Y_l^m = \mu_l^*(\alpha) Y_l^m.
	\]
Using \eqref{eqn: Yukawa Spherical Harmonic Expansion} we find that
	\[
		\mathcal{Y}_{\alpha}Y_l^m = \int_{\partial B_1}\frac{e^{-\alpha|x-y|}}{|x-y|}Y_l^m(y) \dd \mathcal{H}^2(y) = 4\pi \alpha\, i_l(\alpha)k_l(\alpha)Y_l^m
	\]
and hence,
	\[
		\mu_l^*(\alpha) = 4\pi\alpha\, i_l(\alpha)k_l(\alpha).
	\]
Consequently the eigenvalues $\mu_l^{\alpha}$ of $\mathscr{Y}_{\alpha}$ are
	\[
		\mu_l^{\alpha}= 2(\mu_0^*(\alpha) - \mu_l^*(\alpha)).
	\]
	
Let us collect some facts about the eigenvalues $\mu_l^{\alpha}$.
	\begin{lemma}\label{lem: Eigenvalue Y Properties} $\mu_0^{\alpha}=0$, $(\mu_l^{\alpha})_{l\in\\N}$ is increasing in $l$, and 
		\begin{align}
			\mu_1^{\alpha} &=\frac{4\pi e^{-\alpha}}{\alpha^3}\left(e^{\alpha}-(2\alpha^2+2\alpha+1)e^{-\alpha}\right),\label{eqn: mu1}\\
			\mu_2^{\alpha} &=\frac{4\pi e^{-\alpha}}{\alpha^5}\left((3\alpha^2-9)e^{\alpha} + \left((\alpha^2+3\alpha+3)^2-\alpha^4\right)e^{-\alpha}\right).\label{eqn: mu2}
		\end{align}
	In addition, we can write $\mu_1^{\alpha}$ in terms of $\nl(B_1)$ and its derivatives as
		\begin{equation}\label{eqn: mu1 rewrite}
			\mu_1^{\alpha} = \frac{10\nl(B_1)+8\alpha\nl'(B_1)+\alpha^2\nl''(B_1)}{\Pe(B_1)}.
		\end{equation}
	\end{lemma}
\begin{proof} By definition
	\[
		\mu_l^*(\alpha)=4\pi\alpha i_l(\alpha)k_l(\alpha) = 4\pi I_{l+1/2}(\alpha)K_{l+1/2}(\alpha).
	\]
Next, \cite[Theorem 1]{Segura} gives us that for any $\zeta\geq 0$ and $\alpha>0$,
	\[
		\frac{I_{\zeta+1/2}(\alpha)}{I_{\zeta-1/2}(\alpha)} < \frac{\alpha}{\zeta+\sqrt{\zeta^2+\alpha^2}} < \frac{K_{\zeta-1/2}(\alpha)}{K_{\zeta+1/2}(\alpha)},
	\]
that is, $I_{\zeta+1/2}(\alpha)K_{\zeta+1/2}(\alpha)<I_{\zeta-1/2}(\alpha)K_{\zeta-1/2}(\alpha)$. Taking $\zeta=l+1$ yields that $(\mu_l^*)_{l\in\N}$ is strictly decreasing; hence, $(\mu_l^{\alpha})_{l\in\\N}$ is strictly increasing in $l$. Note that, by \cite[Eq. (14.196)]{Arfken} we have 
	\begin{align*}
		\mu_0^*(\alpha)&=\frac{2\pi e^{-\alpha}}{\alpha}\left(e^{\alpha}-e^{-\alpha}\right),\\
		\mu_1^*(\alpha)&=\frac{2\pi e^{-\alpha}}{\alpha^3}\left((\alpha^2-1)e^{\alpha}+(\alpha+1)^2e^{-\alpha}\right),\\
		\mu_2^*(\alpha)&=\frac{2\pi e^{-\alpha}}{\alpha^5}\left((\alpha^4-3\alpha^2+9)e^{\alpha}-(\alpha^2+3\alpha+3)^2 e^{-\alpha}\right),
	\end{align*}
and the formulas \eqref{eqn: mu1} and \eqref{eqn: mu2} follow. As for \eqref{eqn: mu1 rewrite} we can use Lemma~\ref{lem: Properties of Yukawa B_1}, but there is an alternative proof without explicitly knowing $\nl(B_1)$ and its derivatives. As in \cite[Proposition 8.4]{FFMMM} we can appeal to the (almost) homogeneity of $\nl(B_r)$. On the one hand
	\begin{align*}
		\frac{\dd}{\dd r}\bigg|_{r=1} \nl(B_r) &= \frac{\dd}{\dd r}\bigg|_{r=1} \int_{B_1}\int_{B_1} r^5Y_{\alpha r}(|x-y|) \dd y \dd x \\
&= 5\nl(B_1)-\alpha \int_{B_1}\int_{B_1} e^{-\alpha|x-y|} \dd x \dd y = 5\nl(B_1)+\alpha\nl'(B_1).
	\end{align*}
On the other hand, notice
	\begin{multline*}
		\div(xY_{\alpha}(x)) = \langle \nabla Y_{\alpha}(x), x\rangle + Y_{\alpha}(x)\div(x) \\
		= \left\langle \frac{\nabla e^{-\alpha |x|}}{|x|} - \frac{e^{-\alpha |x|}}{|x|^3}x , x\right\rangle + 3Y_{\alpha}(x) = \left\langle \nabla  e^{-\alpha |x|}, \frac{x}{|x|}\right\rangle +2 Y_{\alpha}(x),
	\end{multline*}
so that by the divergence theorem,
	\begin{align*}
		\frac{\dd}{\dd r}\bigg|_{r=1} \nl(B_r) &= 2\int_{\partial B_1}\int_{B_1} \frac{e^{-\alpha|x-y|}}{|x-y|} \dd x \dd\mathcal{H}^2(y) \\
&=\int_{\partial B_1}\int_{B_1} \div_x\left((x-y)Y_{\alpha}(x-y)\right)-\left\langle \nabla_x e^{-\alpha|x-y|}, \frac{x-y}{|x-y|}\right\rangle \dd x \dd\mathcal{H}^2(y)\\
&=\int_{\partial B_1}\int_{\partial B_1}\langle x-y, y\rangle Y_{\alpha}(x-y) \dd\mathcal{H}^2(x)d\mathcal{H}^2(y) \\
&\qquad\qquad\qquad\qquad\qquad\qquad +\alpha \int_{B_1}\int_{\partial B_1} e^{-\alpha|x-y|} \dd\mathcal{H}^2(y)\dd x\\
&=\frac{1}{2}\int_{\partial B_1}\int_{\partial B_1} |x-y|e^{-\alpha|x-y|} \dd x\dd y + \alpha \int_{B_1}\int_{\partial B_1}e^{-\alpha|x-y|} \dd\mathcal{H}^2(y)\dd x.
\end{align*}
We recognize the first term as $1/2\mu_1^{\alpha}\Pe(B_1)$. Indeed, observe that the coordinate functions $x_1,x_2,x_3$ are harmonic on $\partial B_1$ and so for $i=1,2,3$ we have $\mathscr{Y}_{\alpha}x_i = \mu_1^{\alpha}x_i$. Then, applying \eqref{eqn: alpha Norm to Y Operator} to each of the $x_i$ and summing the contributions gives for $x\in \partial B_1$ that
\begin{multline}\label{eqn: mu1 Computation}
	\mu_1^{\alpha}\Pe(B_1)= \sum_{i=1}^3 \int_{\partial B_1}x_i \mathscr{Y}_{\alpha}x_i \dd x= \sum_{i=1}^3 [x_i]_{\alpha}^2 =[x]_{\alpha}^2\\
	=\int_{\partial B_1}\int_{\partial B_1} |x-y|e^{-\alpha|x-y|} \dd\mathcal{H}^2(y)\dd\mathcal{H}^2(x).
\end{multline}
Hence, 
	\begin{align*}
		\mu_1^{\alpha}= \frac{10\nl(B_1)+2\alpha\nl'(B_1)}{\Pe(B_1)}-\frac{2\alpha}{\Pe(B_1)} \int_{B_1}\int_{\partial B_1}e^{-\alpha|x-y|} \dd\mathcal{H}^2(y)\dd x.
	\end{align*}
For the last term, observe that differentiating
	\[
		2\int_{\partial B_1}\int_{B_1} \frac{e^{-\alpha|x-y|}}{|x-y|} \dd x \dd \mathcal{H}^2(y) = \frac{\dd}{\dd r}\bigg|_{r=1} \nl(B_r) = 5\nl(B_1)+\alpha \nl' (B_1)
	\] 
in $\alpha$ yields
	\[
		-2\int_{\partial B_1}\int_{B_1} e^{-\alpha|x-y|} \dd x\dd \mathcal{H}^2(y) = 6\nl'(B_1)+\alpha \nl''(B_1),
	\]
and \eqref{eqn: mu1 rewrite} follows.
\end{proof}

Now, we recall the second variation formulas of $B_1$ for the perimeter and the Yukawa potential obtained in \cite{FFMMM}.
	\begin{theorem}[{cf. \cite[Theorem 6.1]{FFMMM}}]\label{thm:second_var}
		Let $X\in C_c^{\infty}(\mathbb{R}^3;\mathbb{R}^3)$ be a volume preserving flow on $B_1$. Denoting $\zeta = \langle X,\nu_{B_1}\rangle$ then
		\begin{align*}
			\delta^2 \Pe(B_1)[X] &= \int_{\partial B_1} |\nabla_{\tau}\zeta(z)|^2 \dd\mathcal{H}^2(z) - 2\int_{\partial B_1} \zeta(z)^2 \dd\mathcal{H}^2(z),\\
			\delta^2 \nl(B_1)[X] &= -[\zeta]_{\alpha}^2+ \int_{\partial B_1} c_{Y_{\alpha},\partial B_1}^2(z) \zeta(z)^2 \dd\mathcal{H}^2(z),
		\end{align*}
where
		\[
			c_{Y_{\alpha},\partial B_1}^2(x):=\int_{\partial B_1}Y_{\alpha}(x-y)|\nu_{B_1}(x)-\nu_{B_1}(y)|^2 \dd\mathcal{H}^2(y).
		\]
	\end{theorem}
	
Let us determine $c_{Y_{\alpha},\partial B_1}^2$ now. Observe first that it is a constant in $x$. Then since $\nu_{B_1}(x)=x$, by integrating over $\partial B_1$ and applying \eqref{eqn: mu1 Computation} we have
	\[
		c_{Y_{\alpha},\partial B_1}^2\Pe(B_1) = \int_{\partial B_1}\int_{\partial B_1}|x-y|e^{-\alpha|x-y|} \dd\mathcal{H}^2(y) \dd\mathcal{H}^2(x) = \mu_1^{\alpha}\Pe(B_1).
	\]
That is, $c_{Y_{\alpha}, \partial B_1}^2 = \mu_1^{\alpha}$. Owing to this, we define the quadratic functionals
	\begin{gather*}
		\mathcal{QP}(u)\defeq \int_{\partial B_1}|\nabla_{\tau}u(z)|^2 \dd\mathcal{H}^2(z) - \lambda_1\int_{\partial B_1} u(z)^2 \dd\mathcal{H}^2(z),  \\
		\mathcal{QN_{\alpha}}(u) \defeq [u]_{\alpha}^2- \mu_1^{\alpha} \int_{\partial B_1} u(z)^2 \dd\mathcal{H}^2(z),
	\end{gather*}
with $\lambda_l=l(l+1)$. These functionals are defined on the space
	\[
		\tilde{H}^1(\partial B_1)\defeq \left\lbrace u\in H^1(\partial B_1) \colon \int_{\partial B_1} u(x) \dd \mathcal{H}^2(x)=0\right\rbrace
	\]
which generalizes the space of normal components of volume preserving flows $X$ on $B_1$. With this we have the following (analogous to \cite[Proposition 7.2]{FFMMM})
	\begin{proposition}\label{prop:second_var}
		Let $\alpha\geq 0$ and $\beta> 0$. Then for every $u\in \tilde{H}^1(\partial B_1)$ it holds
		\begin{equation}\label{eqn:expansion}
			\mathcal{QP}(u)-\beta \mathcal{Q\nl}(u) = \sum_{l=2}^{\infty}\sum_{m=-l}^{l}\Big((\lambda_l - \lambda_1)-\beta(\mu_l^{\alpha} - \mu_1^\alpha)\Big) |a_l^m(u)|^2,
		\end{equation}
where $a_l^m(u) = \int_{\partial B_1}u Y_l^{m*} \dd\mathcal{H}^2$. 
	\end{proposition}
	
In particular, $\mathcal{QP}(u)-\beta \mathcal{Q\nl}(u)\geq 0$ for all $u\in \tilde{H}^1(\partial B_1)$ if and only if  $\beta \in (0,\beta_*(\alpha)]$, where
	\begin{equation}\label{eqn: beta* Definition}
		\beta_*(\alpha)\defeq \inf_{l\geq 2}\, \frac{\lambda_l-\lambda_1}{\mu_l^{\alpha}-\mu_1^{\alpha}}.
	\end{equation}
By Proposition~\ref{prop:monotonicity_beta*}, the sequence of eigenvalue ratios is strictly increasing, and the infimum above is achieved when $l=2$. Moreover, since $\lambda_l = l(l+1)$, by \eqref{eqn: mu1} and \eqref{eqn: mu2} we have 
		\begin{equation}\label{eqn: beta* Value}
			\beta_*(\alpha) = \frac{\lambda_2-\lambda_1}{\mu_2^{\alpha}-\mu_1^{\alpha}}  = \frac{\alpha^5 e^{\alpha}}{\pi\big[(2\alpha^2-9)e^{\alpha} + \left((\alpha^2+3\alpha+3)^2+\alpha^2(\alpha+1)^2\right)e^{-\alpha}\big]}.
		\end{equation}

Now we are ready to prove the result that determines $\beta_*(\alpha)$ as the critical stability threshold.

\begin{proof}[Proof of Theorem~\ref{thm: Stable Critical Threshold Beta}]
If $X\in C^\infty_c(\R^3;\R^3)$ induces a volume-preserving flow, then by Theorem~\ref{thm:second_var} the second variation of $\ene$ at $B_1$ with respect to $X$ is
	\[
		\delta^2\ene(B_1)[X] = \mathcal{QP}(X\cdot\nu_{B_1})-\beta \mathcal{Q\nl}(X\cdot\nu_{B_1}).
	\]
This means that $B_1$ is a volume-constrained stable set for $\ene$ if and only if
	\begin{equation} \label{eqn:stability}
		\mathcal{QP}(u)-\beta \mathcal{Q\nl}(u)\geq 0 \qquad \text{for every }u\in C^\infty(\partial B_1) \text{ with }\int_{\partial B_1}u\dd\Htwo=0.
	\end{equation} 
Note that the condition \eqref{eqn:stability} implies that $\delta^2\ene(B_1)[X]\geq0$ for every vector field $X$ inducing a volume-preserving flow, since for every such vector field it must be $\int_{\partial B_1} X\cdot\nu_{B_1}\dd\Htwo=0$. The converse implication can be proved by arguing as in \cite[Proof of Theorem~7.1]{FFMMM}. So, the theorem follows by Proposition~\ref{prop:second_var}.
\end{proof}

\section{Rigidity of \texorpdfstring{$L^1$}{L1}-local minimizers}

In this section we prove Theorem~\ref{thm: L1 Local Minimality}. As in Section \ref{sec:min_balls}, by Remark \ref{rem:master} it suffices to prove
	\begin{theorem}\label{thm: L1 Local Minimality Beta} 
		Let $\alpha >0$. If $\beta < \beta_*(\alpha)$ then $B_1$ is a local volume-constrained minimizer of $\ene$, that is there exists $\delta(\alpha,\beta)>0$ such that if $|E\Delta B_1|\leq \delta$ and $|E|=|B_1|$ then
		\[
			\ene(B_1)\leq \ene(E)
		\]
and moreover if equality holds then $|E\Delta B_1|=0$ up to translation. If instead $\beta>\beta_*(\alpha)$, then $B_1$ is not a local volume-constrained minimizer of $\ene$.
	\end{theorem}
	
\begin{proof}[Proof of Theorem~\ref{thm: L1 Local Minimality}] 
	For $m>0$ and $\alpha \geq 0$ we have that
	\[
		\calF_{\alpha}(B) \leq \calF_{\alpha}(F) \qquad \text{whenever} \quad |F|=m \ \text{and} \ |F\Delta B[m]|\leq \epsilon_*m
	\]
	if and only if
	\begin{equation}\label{eqn: L1 Local Minimality}
		\calE_{\alpha',\beta}(B) \leq \calE_{ \alpha', \beta}(E) \qquad \text{whenever} \quad |E|=|B_1| \ \text{and} \ |E\Delta B_1|\leq \epsilon_*|B_1|
	\end{equation}
	with $\alpha' = \alpha\beta^{1/3}$ and $\beta = m/|B_1|$. By Theorem~\ref{thm: L1 Local Minimality Beta}, if $\beta < \beta_*(\alpha')$, i.e.
	\[
		m < |B_1|\beta_*\left(\alpha(m/|B_1|)^{1/3}\right),
	\]
or equivalently $m<m_*(\alpha)$, then there exists $\delta(\alpha,m)>0$ such that if $\epsilon_*<\delta/|B_1|$ then \eqref{eqn: L1 Local Minimality} holds.  
\end{proof}

Before we prove Theorem~\ref{thm: L1 Local Minimality Beta}, we obtain a refinement of Lemma~\ref{lem:nonloc_L2_control}. To this end, for $u\in H^1(\partial B_1)$, recall that $[u]_{\alpha}^2$ is defined in \eqref{eqn: alpha Norm to Y Operator} by
	\[
		[u]_{\alpha}^2 = \int_{\partial B_1}\int_{\partial B_1} |u(x)-u(y)|^2 Y_{\alpha}(x-y) \dd\mathcal{H}^2(y)\dd\mathcal{H}^2(x).
	\]
Then for nearly spherical sets we have the following control on the nonlocal term of $\ene$.

	\begin{lemma}\label{lem: Fuglede Yukawa Potential} 
		There exist $C>0$ and $t_0>0$, independent of $\alpha$, so that if $E\subset \mathbb{R}^3$ is an open set with $|E|=|B_1|$ and 
	\[
		\partial E = \Big\lbrace(1+t\,u(x))x \, : \, x\in \partial B_1\Big\rbrace
	\]
for some $u\in C^1(\partial B_1)$ with $\|u\|_{C^1(\partial B_1)}\leq 1/2$ and $t\in (0,2t_0)$, then
	\[
		\nl(B_1)-\nl(E)\leq \left([u]_{\alpha}^2-\mu_1^{\alpha}\|u\|_{L^2(\partial B_1)}^2\right) \frac{t^2}{2}+\left(\frac{C}{2}[u]_{\alpha}^2+B(\alpha) \|u\|_{L^2(\partial B_1)}^2\right)t^3 
	\]
where $B(\alpha)$ is a  positive function defined by
	\begin{align*}
		B(\alpha):=&\max\left\lbrace B_1(\alpha)+B_2(\alpha),B_2(\alpha)\right\rbrace, \quad B_1(\alpha):=\frac{15\alpha^2\nl''(B_1)+5\alpha^3\nl'''(B_1)}{24\Pe(B_1)},\\
		&B_2(\alpha):=-\frac{120\alpha \nl'(B_1)+90\alpha^2 \nl''(B_1)+25\alpha^3\nl'''(B_1)}{48\Pe(B_1)}.
	\end{align*}
	\end{lemma}
	
\begin{proof} Introduce, for $r,\rho,\theta\geq 0$, the function
	\[
		f_{\theta}^{\alpha}(r,\rho) = \frac{r^2\rho^2 e^{-\alpha (|r-\rho|^2+r\rho\theta^2)^{1/2}}}{(|r-\rho|^2+r\rho\theta^2)^{1/2}}.
	\]
With this notation we have
	\[
		\nl(E) = \int_{\partial B_1}\int_{\partial B_1}\int_0^{1+tu(x)}\int_0^{1+tu(y)} f_{|x-y|}^{\alpha}(r,\rho) \dd\rho \dd r \dd\mathcal{H}^2(y) \dd\mathcal{H}^2(x).
	\]
We can rearrange the integrals to instead write
	\begin{multline}\label{eqn: Potential Exact}
		\nl(E) = \int_{\partial B_1}\int_{\partial B_1} \int_0^{1+t u(x)}\int_0^{1+tu(x)} f_{|x-y|}^{\alpha}(r,\rho) \dd\rho \dd r \dd\mathcal{H}^2(y)\dd\mathcal{H}^2(x) \\
 				-\frac{1}{2}\int_{\partial B_1}\int_{\partial B_1} \int_{1+tu(y)}^{1+t u(x)}\int_{1+tu(y)}^{1+tu(x)} f_{|x-y|}^{\alpha}(r,\rho) \dd\rho \dd r \dd\mathcal{H}^2(y)\dd\mathcal{H}^2(x).
	\end{multline}	
We start by estimating the first term. Observe that unlike in \cite[Lemma 5.3]{FFMMM} the function $f_{\theta}^{\alpha}$ is not homogeneous. In particular, we instead have for $\lambda>0$ that
	\[
		f_{\theta}^{\alpha}(\lambda r, \lambda \rho) = \lambda^3 f_{\theta}^{\alpha \lambda} (r,\rho).
	\]
Consequently, after a change of variables, we have 
	\begin{multline}\label{eqn: Potential Estimate Term 1.1}
		\int_{\partial B_1} \int_0^{1+t u(x)}\int_0^{1+tu(x)} f_{|x-y|}^{\alpha}(r,\rho) \dd\rho \dd r \dd\mathcal{H}^2(y) \\ = 
			 (1+tu(x))^5 \int_{\partial B_1} \int_0^1\int_0^1 f_{|x-y|}^{\alpha(1+tu(x))}(r,\rho) \dd\rho \dd r \dd\mathcal{H}^2(y).
	\end{multline}
Using the inequality $e^{-s} \geq 1-s+s^2/2-s^3/6$ with $s=\alpha (|r-\rho|^2+r\rho|x-y|^2)^{1/2} u(x)t$ we get
	\begin{align*}
		f_{|x-y|}^{\alpha(1+tu(x))}(r,\rho) &\geq f_{|x-y|}^{\alpha}(r,\rho) + \left(\alpha\frac{\dd}{\dd\alpha} f_{|x-y|}^{\alpha}(r,\rho)\right) u(x)\, t\\
		&\qquad+ \left(\frac{\alpha^2}{2}\frac{\dd^2}{\dd\alpha^2} f_{|x-y|}^{\alpha}(r,\rho)\right) u(x)^2\, t^2+ \left(\frac{\alpha^3}{6}\frac{\dd^3}{\dd\alpha^3} f_{|x-y|}^{\alpha}(r,\rho)\right) u(x)^3\, t^3.
	\end{align*}
We integrate this to get
	\begin{multline}\label{eqn: Potential Estimate Almost Homogeneous Part}
		\int_{\partial B_1} \int_0^1\int_0^1 f_{|x-y|}^{\alpha(1+tu(x))}(r,\rho) \dd\rho \dd r \dd\mathcal{H}^2(y)\\
		 \geq \frac{\nl(B_1)}{\Pe(B_1)} +\frac{\alpha \nl'(B_1)}{\Pe(B_1)}\,u(x)\,t +\frac{\alpha^2\nl''(B_1)}{2\Pe(B_1)}\,u(x)^2\,t^2+\frac{\alpha^3\nl'''(B_1)}{6\Pe(B_1)}\,u(x)^3\,t^3.
	\end{multline}
Indeed, we notice that the following integral is independent of $x$, and so
	\[
		\int_{\partial B_1} \int_0^1\int_0^1 f_{|x-y|}^{\alpha}(r,\rho) \dd\rho \dd r \dd\mathcal{H}^2(y) = \frac{\nl(B_1)}{\Pe(B_1)}.
	\]
Thus combining \eqref{eqn: Potential Estimate Term 1.1} and \eqref{eqn: Potential Estimate Almost Homogeneous Part} we get
	\begin{multline}\label{eqn: Potential Estimate Term 1.2}
		\int_{\partial B_1}\int_{\partial B_1}\int_0^{1+t u(x)}\int_0^{1+tu(x)} f_{|x-y|}^{\alpha}(r,\rho) \dd\rho \dd r \dd\mathcal{H}^2(y)\dd\mathcal{H}^2(x) \\
					\geq \int_{\partial B_1}(1+tu(x))^5\sum_{k=0}^3I_k(\alpha) u(x)^kt^k\dd\mathcal{H}^2(x)
	\end{multline}
where we define $I_k(\alpha)$ by
	\[
		I_k(\alpha) = \frac{\alpha^k \nl^{(k)}(B_1)}{k!\,\Pe(B_1)}.
	\]
By the volume constraint $|E|=|B_1|$ we have
	\begin{equation}\label{eqn: Volume Constraint}
		\int_{\partial B_1} (1+tu(x))^3 \dd\mathcal{H}^2(x) = \Pe(B_1).
	\end{equation}
Applying this and \eqref{eqn: Potential Estimate Term 1.2} in \eqref{eqn: Potential Exact} then yields
	\begin{multline}\label{eqn: Potential First Estimate}
		\nl(B_1)-\nl(E)\leq \frac{t^2}{2}g(t)-\int_{\partial B_1} (1+tu(x))^3\sum_{k=1}^3 I_k(\alpha)u(x)^kt^k\dd\mathcal{H}^2(x) \\
				+\int_{\partial B_1}(1+tu(x))^3(1-(1+tu(x))^2)\sum_{k=0}^3 I_k(\alpha)u(x)^kt^k\dd\mathcal{H}^2(x), 
	\end{multline}
where 
	\[
		g(t) = \int_{\partial B_1}\int_{\partial B_1} \int_{u(y)}^{u(x)}\int_{u(y)}^{u(x)} f_{|x-y|}^{\alpha}(1+tr,1+t\rho) \dd\rho \dd r \dd\mathcal{H}^2(y)\dd\mathcal{H}^2(x).
	\]
To estimate the first term in \eqref{eqn: Potential First Estimate}, observe that
	\[
		g(0) = \int_{\partial B_1}\int_{\partial B_1} |u(x)-u(y)|^2Y_{\alpha}(x-y) \dd\mathcal{H}^2(y)\dd\mathcal{H}^2(x) = [u]_{\alpha}^2.
	\]
Moreover, arguing as in \cite[Appendix D]{FFMMM}, we find a constant $C>0$, independent of $\alpha$, such that for $\tau\in (0,t)$ we have $g(t)=g(0)+g'(\tau)t$ and $|g'(\tau)| \leq Cg(0)$. To estimate the second term in \eqref{eqn: Potential First Estimate}, we first expand $(1+tu(x))^3$ to get
	\begin{equation}\label{eqn: Potential Estimate Term 2.1}
	\begin{aligned}
		&-\int_{\partial B_1}(1+tu(x))^3\sum_{k=1}^3I_k(\alpha)u(x)^kt^k\dd \mathcal{H}^2(x) \\
		&\qquad= -I_1(\alpha)\lbr u\rbr_1\,t-\left(3I_1(\alpha)+I_2(\alpha)\right)\lbr u\rbr_2\,t^2-\left(3I_1(\alpha)+3I_2(\alpha)+I_3(\alpha)\right)\lbr u\rbr_3\,t^3\\
		&\qquad\qquad -\left(I_1(\alpha)+3I_2(\alpha)+3I_3(\alpha)\right)\lbr u\rbr_4\,t^4-\left(I_2(\alpha)+3I_3(\alpha)\right)\lbr u\rbr_5\,t^5-I_3(\alpha)\lbr u\rbr_6\,t^6,
	\end{aligned}
	\end{equation}
where $\lbr u\rbr_k \defeq \int_{\partial B_1} u(x)^k \dd \mathcal{H}^2(x)$.
Next, notice that \eqref{eqn: Volume Constraint} further implies
	\begin{equation}\label{eqn: Volume Constraint Reduction}
		-t\lbr u\rbr_1 = t^2\lbr u\rbr_2+\frac{t^3}{3}\lbr u\rbr_3, 
	\end{equation}
and applying this in \eqref{eqn: Potential Estimate Term 2.1} yields
	\begin{equation}\label{eqn: Potential Estimate Term 2.2}
	\begin{aligned}
		&-\int_{\partial B_1}(1+tu(x))^3\sum_{k=1}^3I_k(\alpha)u(x)^kt^k\dd \mathcal{H}^2(x) \\
		&\qquad = -\left(2I_1(\alpha)+I_2(\alpha)\right)\lbr u\rbr_2\,t^2-\left(\frac{8}{3}I_1(\alpha)+3I_2(\alpha)+I_3(\alpha)\right)\lbr u\rbr_3\,t^3\\
		&\ \ \quad\qquad -\left(I_1(\alpha)+3I_2(\alpha)+3I_3(\alpha)\right)\lbr u\rbr_4\,t^4-\left(I_2(\alpha)+3I_3(\alpha)\right)\lbr u\rbr_5\,t^5-I_3(\alpha)\lbr u\rbr_6\,t^6.
	\end{aligned}
	\end{equation}
For the third term in \eqref{eqn: Potential First Estimate}, if $t_0$ is as in Lemma~\ref{lem: Properties of Yukawa B_1}(ii) then we  have for $s\in (-t_0, t_0)$ that
	\[
		\nl(B_1)+\alpha \nl'(B_1)s+\alpha^3\nl'''(B_1)\frac{s^3}{6}\geq 0.
	\]
As $\|u\|_{L^{\infty}(\partial B_1)}\leq 1/2$ and $t<2t_0$, this holds in particular for $s = u(x)t$. Since $\nl''(B_1)>0$ by Lemma~\ref{lem: Properties of Yukawa B_1}(iii), applying the above gives $\sum_{k=0}^3I_k(\alpha)u(x)^kt^k>0$. Thus, we can estimate
	\begin{align*}
		\int_{\partial B_1}&(1+tu(x))^3\big(1-(1+tu(x))^2\big)\sum_{k=0}^3 I_k(\alpha)u(x)^kt^k\dd\mathcal{H}^2(x)  \\
						   &\qquad\leq \int_{\partial B_1}\big(-2tu(x)-7t^2u(x)^2-9t^3u(x)^3\big)\sum_{k=0}^3 I_k(\alpha)u(x)^kt^k\dd\mathcal{H}^2(x).
	\end{align*}
Expanding the right-hand side then yields
	\begin{align*}
		\int_{\partial B_1}&(1+tu(x))^3\big(1-(1+tu(x))^2\big)\sum_{k=0}^3 I_k(\alpha)u(x)^kt^k\dd\mathcal{H}^2(x)  \\
						   &\qquad\leq -2I_0(\alpha)\lbr u \rbr_1\,t -\left(7I_0(\alpha)+2I_1(\alpha)\right) \lbr u \rbr_2\,t^2-\left(9I_0(\alpha)+7I_1(\alpha)+2I_2(\alpha)\right) \lbr u \rbr_3\,t^3\\
						   &\qquad\quad  -\left(9I_1(\alpha)+7I_2(\alpha)+2I_3(\alpha)\right)\lbr u \rbr_4\,t^4 -\left(9I_2(\alpha)+7I_3(\alpha)\right)\lbr u \rbr_5\, t^5-9I_3(\alpha)\lbr u \rbr_6\, t^6.
	\end{align*}
We again apply \eqref{eqn: Volume Constraint Reduction} to obtain
	\begin{equation}\label{eqn: Potential Estimate Term 3.1}
		\begin{aligned}
			&\int_{\partial B_1}(1+tu(x))^3\big(1-(1+tu(x))^2\big)\sum_{k=0}^3 I_k(\alpha)u(x)^kt^k\dd\mathcal{H}^2(x)  \\
				&\qquad\leq  -\left(5I_0(\alpha)+2I_1(\alpha)\right) \lbr u \rbr_2\,t^2-\left(\frac{25}{3}I_0(\alpha)+7I_1(\alpha)+2I_2(\alpha)\right) \lbr u \rbr_3\,t^3\\
						   &\qquad\qquad  -\left(9I_1(\alpha)+7I_2(\alpha)+2I_3(\alpha)\right)\lbr u \rbr_4\,t^4 \\
						   &\qquad\qquad\qquad\qquad-\left(9I_2(\alpha)+7I_3(\alpha)\right)\lbr u \rbr_5\, t^5-9I_3(\alpha)\lbr u \rbr_6\, t^6.
		\end{aligned}
	\end{equation}
Combining \eqref{eqn: Potential First Estimate}, $g(t)=g(0)+Cg(0)t$, \eqref{eqn: Potential Estimate Term 2.2}, \eqref{eqn: Potential Estimate Term 3.1}, and $\lbr u\rbr_k \leq 1/2^{k-2}\|u\|_{L^2(\partial(B_1))}^2$ for $k\geq 3$, we finally have
	\begin{multline*}
		\nl(B_1)-\nl(E)\leq \left([u]_{\alpha}^2-\left(10I_0(\alpha)+8I_1(\alpha)+2I_2(\alpha)\right)\|u\|_{L^2(\partial B_1)}^2\right) \frac{t^2}{2} \\
		+\left(\frac{C}{2}[u]_{\alpha}^2+\frac{1}{2}\left(\frac{25}{3}I_0(\alpha)+\frac{29}{3}I_1(\alpha)+5I_2(\alpha)+I_3(\alpha)\right)^- \|u\|_{L^2(\partial B_1)}^2\right)t^3 \\
		+ \frac{1}{4}\left(10I_1(\alpha)+10I_2(\alpha)+5I_3(\alpha)\right)^-\|u\|_{L^2(\partial B_1)}^2\,t^4 \\
		+ \frac{1}{8}\left(10I_2(\alpha)+10I_3(\alpha)\right)^-\|u\|_{L^2(\partial B_1)}^2\,t^5+ \frac{1}{16}\left(10I_3(\alpha)\right)^-\|u\|_{L^2(\partial B_1)}^2\,t^6.
	\end{multline*}
Now, by Lemma~\ref{lem: Properties of Yukawa B_1}(iii),
	\begin{multline*}
		\frac{25}{3}I_0(\alpha)+\frac{29}{3}I_1(\alpha)+5I_2(\alpha)+I_3(\alpha) \\
		= \frac{50\nl(B_1)+58\alpha \nl'(B_1)+15\alpha^2\nl''(B_1)+\alpha^3\nl'''(B_1)}{6\Pe(B_1)}>0,
	\end{multline*}
and
	\begin{align*}
		10I_1(\alpha)+10I_2(\alpha)+5I_3(\alpha)&=\frac{60\alpha \nl'(B_1)+30\alpha^2\nl''(B_1)+5\alpha^3\nl'''(B_1)}{6\Pe(B_1)} <0,\\
		10I_3(\alpha)&=\frac{5\alpha^3\nl'''(B_1)}{3\Pe(B_1)}<0.
	\end{align*}
So we have
	\begin{multline*}
		\nl(B_1)-\nl(E)\leq \left([u]_{\alpha}^2-\left(10I_0(\alpha)+8I_1(\alpha)+2I_2(\alpha)\right)\|u\|_{L^2(\partial B_1)}^2\right) \frac{t^2}{2} +\frac{C}{2}[u]_{\alpha}^2\,t^3 \\
		- \frac{10I_1(\alpha)+10I_2(\alpha)+5I_3(\alpha)}{4}\|u\|_{L^2(\partial B_1)}^2\,t^4 \\
		+ \frac{\left(5I_2(\alpha)+5I_3(\alpha)\right)^-}{4}\|u\|_{L^2(\partial B_1)}^2\,t^5- \frac{5I_3(\alpha)}{8}\|u\|_{L^2(\partial B_1)}^2\,t^6.
	\end{multline*}
Up to decreasing $t_0$, we may assume $t<1$; hence,
 	\begin{multline*}
		\nl(B_1)-\nl(E)\leq \left([u]_{\alpha}^2-\left(10I_0(\alpha)+8I_1(\alpha)+2I_2(\alpha)\right)\|u\|_{L^2(\partial B_1)}^2\right) \frac{t^2}{2} \\
		+\left(\frac{C}{2}[u]_{\alpha}^2-\frac{1}{8}\left(20I_1(\alpha)+20I_2(\alpha)+15I_3(\alpha)-\left(10I_2(\alpha)+10I_3(\alpha)\right)^-\right)\|u\|_{L^2(\partial B_1)}^2\right)t^3.
	\end{multline*}
This implies the result since by \eqref{eqn: mu1 rewrite} we have
	\[
		10I_0(\alpha)+8I_1(\alpha)+2I_2(\alpha) = \frac{10\nl(B_1)+8\alpha\nl'(B_1)+\alpha^2\nl''(B_1)}{\Pe(B_1)} = \mu_1^{\alpha},
	\]
and by direct computation
	\begin{multline*}
		-\frac{1}{8}\left(20I_1(\alpha)+20I_2(\alpha)+15I_3(\alpha)-\left(10I_2(\alpha)+10I_3(\alpha)\right)^-\right) \\
	=\frac{1}{8}\max\left\lbrace -20I_1(\alpha)+30I_2(\alpha)+25I_3(\alpha),-20I_1(\alpha)+20I_2(\alpha)+15I_3(\alpha)\right\rbrace =B(\alpha),
	\end{multline*}
as desired.
\end{proof}

Next we prove the following theorem which provides a quantitative bound on the energy deficit.
	\begin{theorem}\label{thm: Quantitative Deficit}
		There exists $C>0$, independent of $\alpha$, such that for every $\alpha\geq 0$ and $\beta \in (0, \beta_*(\alpha))$ there exists $\epsilon_{\beta}(\alpha)>0$ such that if $E$ is a nearly spherical set with $|E|=|B_1|$, $\int_{E}x \dd x=0$, and $\|u\|_{C^1(\partial B_1)}< \epsilon_{\beta}$, then
		\[
			\ene(E)-\ene(B_1)\geq C\left(1-\frac{\beta}{\beta_*(\alpha)}\right)\|u\|_{H^1(\partial B_1)}^2.
		\]
Moreover, $\epsilon_{\beta}(\alpha)$ can be chosen of the form
\[\epsilon_{\beta}(\alpha) = \frac{\epsilon_0}{\ell(\alpha)}\left(1-\frac{\beta}{\beta_*(\alpha)}\right)\]
for some $\epsilon_0>0$ and explicitly computable linear function $\ell(\alpha)$ with $\ell(0)>0$ and $\ell'(0)>0$. In particular, $\epsilon_{\beta}(\alpha)$ does not degenerate as $\alpha\to 0^+$. 
	\end{theorem}
	
\begin{proof}
As in Lemma~\ref{lem: Fuglede Yukawa Potential} we consider $u\in C^1(\partial B_1)$ with $\|u\|_{C^1(\partial B_1)}\leq 1/2$ and $t\in (0,2\epsilon_{\beta})$ so that the set $E\subset \mathbb{R}^3$ with boundary 
	\[
		\partial E = \Big\lbrace (1+tu(x))x \colon x\in \partial B_1\Big\rbrace
	\]
satisfies $|E|=|B_1|$ and $\int_E x \dd x=0$. If $\epsilon_{\beta}$ is small enough, then for some $C>0$
	\begin{multline*}
		\ene(E)-\ene(B_1) \geq \frac{t^2}{2}\left(\mathcal{QP}(u)-\beta \mathcal{QN}_{\alpha}(u)\right) \\
  -Ct^3\left(\int_{\partial B_1}|\nabla_{\tau} u|^2 \dd\mathcal{H}^2+\lambda_1\|u\|_{L^2}^2+\beta \left([u]_{\alpha}^2+\frac{B(\alpha)}{C}\right) \|u\|_{L^2(\partial B_1)}^2\right).
	\end{multline*}	
Recall now from \cite[Eqs. (2.25) and (2.28)]{FFMMM} that for $\epsilon_{\beta}$ small enough we have
	\begin{gather}
		\frac{1}{2}\sum_{l=2}^{\infty}\sum_{m=-l}^l |a_l^m(u)|^2 - \sum_{m=-1}^1 |a_1^m(u)|^2-|a_0(u)|^2 \geq \frac{1}{4}\|u\|_{L^2(\partial B_1)}^2, \label{eqn: Sums Bound L2}	\\
\int_{\partial B_1} |\nabla_{\tau}u|^2 \dd\mathcal{H}^2 - \lambda_1\|u\|_{L^2(\partial B_1)}^2 \geq \frac{1}{4}\int_{\partial B_1} |\nabla_{\tau}u|^2 \dd\mathcal{H}^2+\frac{\lambda_1}{4}\|u\|_{L^2(\partial B_1)}^2 .\label{eqn: Poincare}
	\end{gather}
We estimate first $\mathcal{QP}(u)-\beta\mathcal{QN}_{\alpha}(u)$. By \eqref{eqn:expansion} and the definition of $\beta_*(\alpha)$ in \eqref{eqn: beta* Definition}
	\begin{align*}
		\mathcal{QP}(u)-\beta \mathcal{Q\nl}(u) &= \sum_{l=2}^{\infty}\sum_{m=-l}^{l}\left((\lambda_l - \lambda_1)-\beta(\mu_l^{\alpha} - \mu_1^\alpha)\right) |a_l^m(u)|^2 \\
					&\geq \left(1-\frac{\beta}{\beta_*(\alpha)}\right)\sum_{l=2}^{\infty}\sum_{m=-l}^{l}(\lambda_l - \lambda_1) |a_l^m(u)|^2 .
	\end{align*}
Since
	\[
	 \sum_{l=2}^{\infty}\sum_{m=-l}^{l}(\lambda_l - \lambda_1) |a_l^m(u)|^2 = \int_{\partial B_1}|\nabla_{\tau} u|^2 \dd\mathcal{H}^2 - \lambda_1 \|u\|_{L^2(\partial B_1)}^2,
	\]
by applying \eqref{eqn: Poincare} we therefore have
	\begin{equation}\label{eqn: Functional Difference Bound}
		\mathcal{QP}(u)-\beta \mathcal{Q\nl}(u) \geq  \frac{1}{4}\left(1-\frac{\beta}{\beta_*(\alpha)}\right)\left(\int_{\partial B_1}|\nabla_{\tau} u|^2 \dd\mathcal{H}^2 + \lambda_1 \|u\|_{L^2(\partial B_1)}^2\right).
	\end{equation}
In particular this implies $\mathcal{QP}(u)\geq \beta \mathcal{QN}_{\alpha}(u)$. We now control the cubic terms. From Lemma~\ref{lem: Eigenvalue Y Properties} we have
	\[
		\frac{\mu_2^{\alpha}}{\mu_1^{\alpha}} = \frac{(3\alpha^2-9)e^{\alpha} + \left((\alpha^2+3\alpha+3)^2-\alpha^4\right)e^{-\alpha}}{\alpha^2e^{\alpha}-(2\alpha^4+2\alpha^3+\alpha^2)e^{-\alpha}} \geq \frac{6}{5} 
	\]
This follows since the above is equivalent to
	\[
		(3\alpha^2-15)e^{2\alpha} \geq -15-30\alpha-27\alpha^2-14\alpha^3-4\alpha^4,
	\]
and noting that the right-hand side is the fourth order Taylor expansion of the left-hand side with a positive remainder. Then, observe as in \cite[Eq. (8.9)]{FFMMM} that, from \eqref{eqn: Sums Bound L2}, the above, and Lemma~\ref{lem: Eigenvalue Y Properties}, we have
	\[
\mu_1^{\alpha}\|u\|_{L^2(\partial B_1)} \leq 2\mu_1^{\alpha}\sum_{l=2}^{\infty}\sum_{m=-l}^l|a_l^m(u)|^2 \leq 10\sum_{l=2}^{\infty}\sum_{m=-l}^l(\mu_l^{\alpha}-\mu_1^{\alpha})|a_l^m(u)|^2 = 10\mathcal{QN}_{\alpha}(u).
	\]
We are thus reduced to controlling $B(\alpha)$ in terms of $\mu_1^{\alpha}$. From Lemma~\ref{lem: Properties of Yukawa B_1}(i) and \eqref{eqn: mu1}, we deduce
	\[
		B(\alpha) \leq \frac{25}{12}(\alpha+1)\mu_1^{\alpha},
	\]
since in particular for $\alpha$ sufficiently large $B(\alpha)=B_2(\alpha)$. Hence, for some $c_1,c_2>0$
	\begin{align}\label{eqn: Cubic Bound}
		\beta \left([u]_{\alpha}^2+\frac{B(\alpha)}{C} \|u\|_{L^2(\partial B_1)}^2\right) &= \beta\mathcal{QN}_{\alpha}(u)+\beta \left(\mu_1^{\alpha}+\frac{B(\alpha)}{C}\right) \|u\|_{L^2(\partial B_1)}^2\nonumber \\
					&\leq  \left(c_1\alpha+c_2\right)\beta \mathcal{QN}_{\alpha}(u) \leq  \left(c_1\alpha+c_2\right)\mathcal{QP}(u).
	\end{align}
Again applying \eqref{eqn: Poincare} we have
	\[
	\mathcal{QP}(u) = \int_{\partial B_1} |\nabla_{\tau} u|^2 \dd \mathcal{H}^2-\lambda_1 \|u\|_{L^2(\partial B_1)}^2 \geq \frac{1}{4}\left(\int_{\partial B_1} |\nabla_{\tau} u|^2 \dd \mathcal{H}^2+\lambda_1 \|u\|_{L^2(\partial B_1)}^2\right).
	\]
Combining this with \eqref{eqn: Functional Difference Bound} and \eqref{eqn: Cubic Bound} we have thus shown
	\[
		\ene(E)-\ene(B_1)  
				\geq \left(\frac{t^2}{8}\left(1-\frac{\beta}{\beta_*(\alpha)}\right)-\ell(\alpha)t^3\right)\left(\int_{\partial B_1}|\nabla_{\tau} u|^2 \dd\mathcal{H}^2 + \lambda_1 \|u\|_{L^2(\partial B_1)}^2\right)
	\]
for some linear function $\ell(\alpha)$ with $\ell(0), \ell'(0)>0$. Therefore, for sufficiently small $t$ depending on $\alpha$, we have
	\[
		\ene(E)-\ene(B_1)\geq ct^2\left(1-\frac{\beta}{\beta_*(\alpha)}\right)\|u\|_{H^1(\partial B_1)}^2,
	\]
as desired.
\end{proof}
We now prove Theorem \ref{thm: L1 Local Minimality Beta}. The core of the proof is essentially the same as in \cite[Lemma 8.5]{FFMMM}, but with some new ideas to overcome the lack of homogeneity in $\nl$.

\begin{proof}[Proof of Theorem~\ref{thm: L1 Local Minimality Beta}] Since minimality implies stability, if $B_1$ is a local volume-constrained minimizer of $\ene$ then $\beta \leq \beta_*(\alpha)$. So we only need to prove that if $\beta < \beta_*(\alpha)$ then $B_1$ is a local volume-constrained minimizer of $\ene$. Supposing not, then there exists $\beta<\beta_*(\alpha)$ and a sequence $\{E_i\}_{i=1}^{\infty}$ of sets of finite perimeter with 
	\begin{equation}\label{eqn: Contradiction Hypothesis}
	|E_i|=|B_1| \quad \text{and} \quad |E_i\Delta B_1|\to 0, \quad \text{but} \quad \ene(E_i)\leq \ene(B_1).
	\end{equation}
The idea now is to use a selection principle type argument to generate a sequence $F_i$ with better regularity that allows us to apply Theorem~\ref{thm: Quantitative Deficit} to derive a contradiction.

\textit{Step one:} We first show that the sequence $\{E_i\}_{i=1}^{\infty}$ above can be chosen such that $E_i\subset B_R$ for some $R>0$. Fix $\eta>0$, to be chosen later, so that for $i$ large enough we have $|E_i\Delta B_1|<\eta$. By the truncation lemma \cite[VI. 14]{Almgren}, see also \cite[Lemma 29.12]{Mag}, there exists $c_1,c_2>0$ and $r_i>0$ such that $1\leq r_i \leq 1+c_1\eta^{1/3}$ and
	\begin{equation}\label{eqn: Truncation}
		\Pe(E_i\cap B_{r_i}) \leq \Pe(E_i) - \frac{|E_i\setminus B_{r_i}|}{c_2\eta^{1/3}}.
	\end{equation}
\indent Now let $\lambda_i$ be such that $E_i':=\lambda_i(E_i\cap B_{r_i})$ satisfies $|E_i'|=|B_1|$. We aim to show $\{E_i'\}_{i=1}^{\infty}$ still satisfies \eqref{eqn: Contradiction Hypothesis}, while being uniformly bounded in some ball. Since $r_i\to 1^+$ and $|E_i\Delta B_1|\to 0$, it follows that $\lambda_i\to 1^+$ and $|E_i'\Delta B_1|\to 0$. In particular, for $i$ large enough we have that $E_i'\subset B_R$ for $R = 2+c_1\eta^{1/3}$. We now show $\ene(E_i') \leq \ene(B_1)$. Since $\lambda_i \geq 1$, we have $\calN_{\alpha \lambda_i}(E)\leq \nl(E)$. Hence for $\lambda_i$ sufficiently close to $1$ we have
	\begin{align*}
		\ene(E_i') &= \lambda_i^5\left(\lambda_i^{-3}\Pe(E_i\cap B_{r_i})+\beta \calN_{\alpha \lambda_i}(E_i\cap B_{r_i})\right)\\
		&\leq (1+C|E_i \setminus B_{r_i}|)\left(\Pe(E_i \cap B_{r_i}) + \beta \nl(E_i \cap B_{r_i})\right).
	\end{align*}
Next, applying \eqref{eqn: Truncation}, $\nl(E_i\cap B_{r_i})\leq \nl(E_i)$, and \eqref{eqn: Contradiction Hypothesis} we further estimate
	\[
		\ene(E_i') \leq \ene(E_i)+C\left(\ene(B_1) - \frac{1}{c_2\,\eta^{1/3}} \right)|E_i \setminus B_{r_i}|.
	\]
Hence if $\eta$ was chosen so that $\eta^{1/3}<1/(c_2\,\ene(B_1))$ then the term in parentheses is negative and we conclude $\ene(E_i')\leq \ene(E_i)\leq \ene(B_1)$.

\textit{Step two:} Assuming now that the sequence $\{E_i\}_{i=1}^{\infty}$ additionally satisfies $E_i \subset B_R$ with $R>0$ as above, we aim to generate a new sequence $\{F_i\}_{i=1}^{\infty}$ with better regularity, namely that the $F_i$ are nearly spherical. To this end for $\Lambda>0$ we consider the variational problem for each $i\in \mathbb{N}$
	\begin{equation}\label{eqn: Selection Principle}
		\inf\left\lbrace \ene(E)+\Lambda\,|E\Delta E_i| \ | \ E\subset \mathbb{R}^3 \right\rbrace.
	\end{equation}
Note that in the above we assume no volume constraint nor that a priori $E\subset B_R$. We first show these problems admit a minimizer. Indeed, if $F=E\cap B_R$ then $\ene(F) \leq \ene(E)$ by trivial inclusions and $|F\Delta E_i|\leq |E\Delta E_i|$. Hence we can assume $E\subset B_R$, and a minimizer $F_i$ exists with $F_i\subset B_R$. We now claim that there exists a $\Lambda'>0$ such that
	\begin{equation}\label{eqn: Lambda Minimality}
		\Pe(F_i) \leq \Pe(E) + \Lambda'\,|E\Delta F_i|  \qquad \text{for all} \ E\subset \mathbb{R}^3.
	\end{equation}
First if $E\subset \mathbb{R}^3$ then
	\begin{equation}\label{eqn: Nonlocal Lipschitz}
		\nl(E)-\nl(F_i) \leq 2\int_E\int_{E\setminus F_i} \frac{e^{-\alpha|x-y|}}{|x-y|} \dd x\dd y \leq 2|E\setminus F_i| \int_{B_r} \frac{e^{-\alpha|z|}}{|z|}\dd z
	\end{equation}
where $r =(|E|/|B_1|)^{1/3}$. The latter integral is estimated as 
	\[
		\int_{B_r} \frac{e^{-\alpha |z|}}{|z|} \dd z = \frac{4\pi(1-(1+\alpha r)e^{-\alpha r})}{\alpha^2} \leq \frac{4\pi}{\alpha^2}.
	\]
Hence, by \eqref{eqn: Selection Principle} we get
	\[
		\Pe(F_i) \leq \Pe(E)+\beta \nl(E)-\beta \nl(F_i) + \Lambda\, (|E\Delta E_i|-|F_i\Delta E_i|) \leq \Pe(E) + \left(\frac{8\pi\beta}{\alpha^2}+\Lambda\right) |E\Delta F_i|
	\]
which proves \eqref{eqn: Lambda Minimality} for $\Lambda' = 8\pi \beta/\alpha^2+\Lambda$. Next, by testing \eqref{eqn: Selection Principle} with $E_i$ and applying \eqref{eqn: Contradiction Hypothesis} we get
	\begin{equation}\label{eqn: Selection Principle Energy Estimate}
		\ene(F_i) + \Lambda\, |F_i\Delta E_i| \leq \ene(E_i) \leq \ene(B_1)
	\end{equation}
which implies by the triangle inequality, since $|E_i \Delta B_1|\to 0$, that for sufficiently large $i$
	\[
		|F_i\Delta B_1| \leq \frac{2\ene(B_1)}{\Lambda}.
	\]
By standard regularity theory, see for instance \cite[Theorem 26.6]{Mag}, by choosing $\Lambda$ (depending on $\alpha, \beta>0$) sufficiently large  there exists $u_i \in C^{1,\gamma}(\partial B_1)$ where
	\[
		\partial F_i = \Big\lbrace(1+u_i(x))x \, : \, x\in \partial B_1\Big\rbrace \qquad \text{and} \qquad \|u_i\|_{L^{\infty}(\partial B_1)} < \epsilon_{\beta}
	\]
with $\epsilon_{\beta}$ as in Theorem \ref{thm: Quantitative Deficit}.

\textit{Step three:} We conclude by deriving a contradiction by applying Theorem \ref{thm: Quantitative Deficit}. To this end set $G_i = x_i+t_iF_i$ where $x_i$ is such that $\int_{G_i} x \dd x = 0$ and $t_i>0$ is such that $|G_i|=|B_1|$. Up to increasing the value of $\Lambda$, we find $v_i \in C^{1,\gamma}(\partial B_1)$ such that
	\[
		\partial G_i = \Big\lbrace(1+v_i(x))x \, : \, x\in \partial B_1\Big\rbrace \qquad \text{and} \qquad \|v_i\|_{L^{\infty}(\partial B_1)} < \epsilon_{\beta}.
	\]
We can now apply Theorem \ref{thm: Quantitative Deficit} to each $G_i$ to conclude
	\begin{equation}\label{eqn: Selection Principle Contradiction Estimate}
		\ene(B_1)+C\left(1-\frac{\beta}{\beta_*(\alpha)}\right)\|v_i\|_{H^1(\partial B_1)}^2 \leq \ene(G_i) = t_i^2\Pe(F_i) + \beta\,t_i^5 \calN_{t_i\alpha}(F_i).	
	\end{equation}
We first show that $t_i\geq 1$. To this end we prove when $t_i<1$ that
	\begin{equation}\label{eqn: Nonlocal Scaling Selection Principle}
		t_i^3\calN_{t_i\alpha}(F_i) \leq \nl(F_i)+c|F_i\Delta E_i|
	\end{equation}
for some $c>0$ depending only on $\alpha$. By the Lipschitz estimate \eqref{eqn: Nonlocal Lipschitz} we have 
	\[
		|\nl(E_i)-\nl(F_i)| \leq \frac{8\pi^2}{\alpha^2}|F_i\Delta E_i|,
	\]
so that, since $t_i< 1$, 
	\[
		t_i^3\calN_{t_i\alpha}(F_i) \leq \nl(F_i)+ \frac{16\pi^2}{\alpha^2}|F_i\Delta E_i|+|t_i^3\calN_{t_i\alpha}(E_i)-\nl(E_i)| .
	\]
We estimate the latter difference as 
	\[
		|t_i^3\calN_{t_i\alpha}(E_i)-\nl(E_i)| \leq (1-t_i^3)\,\calN_{t_i\alpha}(E_i)+|\calN_{t_i\alpha}(E_i)-\nl(E_i)|.
	\]
For the first term, since $|E_i|=|B_1|$, then by Lemma~\ref{lem:kernel_prop}(iv)
	\[
		(1-t_i^3)\,\calN_{t_i\alpha}(E_i) \leq 3(1-t_i)\calN_0(B_1) \leq \frac{32\pi^2}{5}(1-t_i).
	\]
For the second term, using again $|E_i|=|B_1|$, we have
	\[
		|\calN_{t_i\alpha}(E_i)-\nl(E_i)| \leq \int_{E_i}\int_{E_i} \frac{|e^{-t_i\alpha|x-y|}-e^{-\alpha|x-y|}|}{|x-y|} \dd x \dd y \leq \alpha(1-t_i)|B_1|^2.
	\]
Finally, since $|t_iF_i|=|B_1|$ and $t_i< 1$ then 
	\[
	|F_i\Delta E_i|\geq ||F_i|-|B_1||=|B_1|\left(\frac{1}{t_i^3}-1\right) \geq  |B_1|(1-t_i),
	\]
and \eqref{eqn: Nonlocal Scaling Selection Principle} is proved with $c=16\pi^2/\alpha^2+24\pi/5+4\pi\alpha/3$. Finally, from \eqref{eqn: Selection Principle Energy Estimate} and \eqref{eqn: Nonlocal Scaling Selection Principle} we have
	\begin{align*}
		t_i^2\Pe(F_i) + \beta\, t_i^5\calN_{t_i\alpha}(F_i)  &\leq t_i^2\left(\ene(F_i) + c\beta|F_i\Delta E_i|\right) \\
		&\leq  t_i^2\left(\ene(B_1) +(c\beta-\Lambda )|F_i\Delta E_i|\right) < \ene(B_1)
	\end{align*}
in contradiction to \eqref{eqn: Selection Principle Contradiction Estimate} if $\Lambda>c\beta$, which depends only on $\alpha$ and $\beta$. Hence $t_i \geq 1$ as desired. In this case, we have $\calN_{t_i\alpha}(F_i) \leq \calN_{\alpha}(F_i)$, and applying \eqref{eqn: Selection Principle Energy Estimate} in \eqref{eqn: Selection Principle Contradiction Estimate} gives
	\begin{equation}\label{eqn: Energy Estimate}
		t_i^{-5}\ene(B_1)+Ct_i^{-5}(1-\beta/\beta_*(\alpha))\|v_i\|_{H^1(\partial B_1)}^2 +\Lambda|F_i\Delta E_i|\leq \ene(B_1).
	\end{equation}
Now, if $t_i=1$ then $F_i=E_i$ and $\|v_i\|_{H^1(\partial B_1)}=0$. Hence $E_i=B_1(-x_i)$, a contradiction to \eqref{eqn: Contradiction Hypothesis} unless equality holds. If instead $t_i>1$ then $|F_i\Delta E_i|\geq ||F_i|-|B_1||=|B_1|(1-t_i^{-3})$. Applying this in \eqref{eqn: Energy Estimate} gives
	\[
	\Lambda |B_1|\left(1-\frac{1}{t_i^3}\right)	\leq \ene(B_1)\left(1-\frac{1}{t_i^5}\right)
	\]
where for instance taking $\Lambda>2\ene(B_1)$ yields a contradiction.
\end{proof}

\begin{remark}
	We can deduce the rigidity statement in Theorem~\ref{thm: L1 Local Minimality Beta} by keeping the $\|v_i\|_{H^1(\partial B_1)}$ term in \eqref{eqn: Energy Estimate}. In \cite[Lemma 8.5]{FFMMM}, the authors choose to discard this term; were they to keep it, they would also deduce the same rigidity.
\end{remark}


\appendix
\section{Characterization of the critical threshold \texorpdfstring{$\beta_*(\alpha)$}{β\_*(α)}}\label{sec:appendix}

Here we prove the identity \eqref{eqn: beta* Value} by obtaining the monotonicity of the sequence that defines $\beta_*(\alpha)$ as in \eqref{eqn: beta* Definition}. Similar properties of eigenvalues have been shown in \cite{BCT,FFMMM}, where in the Riesz setting the analogous eigenvalues are ratios of Gamma functions, which exhibit a product recursion across modes and the infimum of ratios can be identified via an induction argument. Such a strategy is not available here. Instead, following the classical literature on potential theory and boundary integral equations (cf. \cite{HW,McLean}), we use a variational characterization of the eigenvalues, which replaces the product structure. To this end, we define
	\[
		\beta_*^{(l)}(\alpha)\defeq \frac{\lambda_l - \lambda_1}{\mu_l^\alpha - \mu_1^\alpha}.
	\]
Our goal is to prove the following result.
	\begin{proposition}\label{prop:monotonicity_beta*}
	 		For any $\alpha>0$ the sequence $(\beta_*^{(l)}(\alpha))_{l\in\N}$ is strictly increasing.
	\end{proposition}
	
Let us define the functionals
	\[
		A(\varphi) \defeq \int_0^\infty \Big((\varphi'(r))^2+\alpha^2\varphi^2(r)\Big)\,r^2\dd r, \qquad B(\varphi)\defeq \int_0^\infty \varphi^2(r)\dd r,
	\]
over the class $\mathscr{A}\defeq \big\{\varphi \in AC_{\loc}((0,\infty)) \colon A(\varphi)+B(\varphi) < \infty \big\}$. For any $\tau\geq 0$, let
	\begin{equation}\label{eqn:G}
		G(\tau) \defeq \inf \Big\{A(\varphi)+\tau B(\varphi) \colon \varphi\in\mathscr{A}, \ \varphi(1)=1\Big\}
	\end{equation}
such that $G(\lambda_l)$ is the minimal screened Dirichlet energy among extensions of $Y_l^m$ to $\R^3$ with trace $Y_l^m$ on the unit sphere. 
	\begin{lemma}\label{lem:characterization_G}
		For any $l\geq 1$, $\disp G(\lambda_l) = \frac{1}{\alpha \, i_l(\alpha) k_l(\alpha)}$,
		and the infimum in \eqref{eqn:G} at $\tau=\lambda_l$ is attained at
			\[
				\varphi_l(r) \defeq \begin{dcases*}
						 				\frac{i_l(\alpha r)}{i_l(\alpha)} &  if $0<r\leq 1$, \\
						 				\frac{k_l(\alpha r)}{k_l(\alpha)}	&  if $r \geq 1$.
									\end{dcases*}
			\]
	\end{lemma}
	
\begin{proof}
Using the behavior of $i_l$ at $r=0$ and $k_l$ at $r=\infty$ we see that $\varphi_l \in \mathscr{A}$. Also, $\varphi_l$ is continuous at $r=1$ and $\varphi_l(1)=1$. Since $i_l$ and $k_l$ solve the modified spherical Bessel equation $x^2u'' +2xu'-\big(x^2+l(l+1)\big)u=0$, $\varphi_l$ solves $\big(r^2\varphi'\big)'=\big(\lambda_l+\alpha^2r^2\big)\varphi$ on $(0,1)$ and $(1,\infty)$, which is exactly the Euler--Lagrange equation of $Q(\varphi)\defeq A(\varphi)+\lambda_l B(\varphi)$.

Now let $\varphi\in\mathscr{A}$ be arbitrary with $\varphi(1)=1$. Define $w\defeq\varphi-\varphi_l$. Then $w\in\mathscr{A}$ with $w(1)=0$, and $Q(\varphi)=Q(\varphi_l)+2I(\varphi_l,w)+Q(w)$, where
	\[
		I(\varphi_l,w) \defeq \int_0^\infty \big(r^2\varphi_l'w'+(\lambda_l+\alpha^2r^2)\varphi_l w\big)\dd r. 
	\]
We claim that $I \equiv 0$. To prove this, first note that since $\varphi_l$ solves the Euler--Lagrange equation on $(0,1)$ and $(1,\infty)$, we have
	\[
		I(\varphi_l,w) = \int_0^1 \big(r^2\varphi_l'w\big)'\dd r + \int_1^\infty \big(r^2\varphi_l'w\big)'\dd r.
	\]
Hence, for $0<\e<1<R$, we get
	\[
		\int_{\e}^R \big(r^2\varphi_l'w\big)'\dd r = R^2 \varphi_l'(R)w(R) - \e^2\varphi_l'(\e)w(\e),
	\]
since $w(1)=0$.

For $0<\e<1/2$ we estimate
	\[
		|w(\e)| \leq |w(1/2)| + \int_{\e}^{1/2} |w'|\dd r \leq |w(1/2)| + A^{1/2}(w)\left(\int_{\e}^{1/2}\frac{\dd r}{r^2}\right)^{1/2} \leq C_w \e^{-1/2}.
	\]
On the other hand, using the properties of $i_l$, $\e^2 \varphi_l'(\e)=O(\e^{l+1})$. Thus $\e^2\varphi_l'(\e)w(\e) = O(\e^{l+1/2})\to 0$ for $l\geq 1$ as $\e\to 0$.

Since $B(w)<\infty$ we have that $w\in L^2((0,\infty))$. So, there exist $R_j \in [j,2j]$ with $w^2(R_j) \leq \frac{1}{j}\int_j^{2j} w^2\dd r \to 0$. Again, a direct calculation using the properties of $k_l$ implies that $R^2\varphi_l'(R)\to 0$ as $R\to \infty$. Thus $R^2 \varphi_l'(R)w(R)\to 0$ as $R\to\infty$. Then the claim follows sending $\e\to 0$ and $R\to \infty$.

As a consequence we obtain that $Q(\varphi)=Q(\varphi_l)+Q(w)$. Since $Q(w) \geq \alpha^2\int_0^\infty w^2 r^2\dd r \geq 0$, and $Q(w)$ vanishes only when $w=0$ we conclude that $\varphi_l$ is the unique minimizer of $Q$.

Therefore, taking $\varphi_l$ instead of $w$,
	\[
		G(\lambda_l) = Q(\varphi_l) = \varphi_l'(1^-) - \varphi_l'(1^+) = \alpha \left(\frac{i_l'(\alpha)}{i_l(\alpha)}-\frac{k_l'(\alpha)}{k_l(\alpha)}\right) =\frac{1}{\alpha\,i_l(\alpha)k_l(\alpha)},
	\]
where we used the fact that 
	\[
		i_l(\alpha)k_l'(\alpha) - i_l'(\alpha)k_l(\alpha) = \frac{1}{\alpha}\big(I_{l+1/2}(\alpha)K_{l+1/2}'(\alpha) - I_{l+1/2}'(\alpha)K_{l+1/2}(\alpha)\big) = -\frac{1}{\alpha^2}
	\]
by direct computation.
\end{proof}

The next lemma establishes the properties of the function $G$; its monotonicity and concavity are key in proving the monotonicity of the eigenvalues.
	\begin{lemma}\label{lem:properties_G}
		The function $G$ is finite, nonnegative and concave on $[0,\infty)$. Moreover we have $G(\lambda_j)<G(\lambda_l)$ whenever $1\leq j < l$.
	\end{lemma}
	
\begin{proof}
Nonnegativity of $G$ clearly follows by its definition. It is also finite since, for example, $A(\varphi_1)+\tau B(\varphi_1) <\infty$. Furthermore the map $\tau \mapsto A(\varphi) +\tau B(\varphi)$ is affine with positive slope. Thus, being a positive infimum of increasing affine functions over a $\tau$-independent index, the function $G$ is increasing and concave. Now let $\varphi_l$ be the minimizer in Lemma~\ref{lem:characterization_G} above. Then
	\[
		G(\lambda_j) \leq A(\varphi_l)+\lambda_j B(\varphi_l) = G(\lambda_l)-(\lambda_l-\lambda_j)B(\varphi_l) < G(\lambda_l),
	\]
as claimed.
\end{proof}

We now prove Proposition \ref{prop:monotonicity_beta*}.

\begin{proof}[Proof of Proposition~\ref{prop:monotonicity_beta*}]
Note that by Lemma~\ref{lem:characterization_G},
	\[
		\mu_l^{\alpha} - \mu_1^{\alpha} = 8\pi \left(\frac{1}{G(\lambda_1)}-\frac{1}{G(\lambda_l)}\right);
	\]
hence, it suffices to prove that
	\begin{equation}\label{eqn:equiv_monotonicity}
		\frac{\lambda_j-\lambda_1}{1/G(\lambda_1)-1/G(\lambda_j)} < \frac{\lambda_l-\lambda_1}{1/G(\lambda_1)-1/G(\lambda_l)}.
	\end{equation}
Let $t\defeq (\lambda_l-\lambda_j)/(\lambda_l-\lambda_1)\in (0,1)$ so that $\lambda_j = t\lambda_1 + (1-t)\lambda_l$. Since $G$ is concave,
	\[
		\frac{1}{G(\lambda_j)} \leq \frac{1}{tG(\lambda_1)+(1-t)G(\lambda_l)}.
	\]
Now, note that
	\begin{multline*}
		\big(tG(\lambda_1)+(1-t)G(\lambda_l)\big)\left(\frac{t}{G(\lambda_1)}+\frac{1-t}{G(\lambda_l)}\right)\\ 
					=t^2 + (1-t)^2 + t(1-t)\left(\frac{G(\lambda_1)}{G(\lambda_l)}+\frac{G(\lambda_l)}{G(\lambda_1)}\right) > 1,
	\end{multline*}
where we used the fact that $x+x^{-1}>2$ for $x\neq 1$. Thus
	\[
		\frac{1}{G(\lambda_j)} < \frac{t}{G(\lambda_1)}+\frac{1-t}{G(\lambda_l)},
	\]
which, in turn, implies that
	\[
		\frac{1}{G(\lambda_1)}-\frac{1}{G(\lambda_j)} > (1-t)\left(\frac{1}{G(\lambda_1)}-\frac{1}{G(\lambda_l)}\right) = \frac{\lambda_j-\lambda_1}{\lambda_l-\lambda_1}\left(\frac{1}{G(\lambda_1)}-\frac{1}{G(\lambda_l)}\right).
	\]
Since all four differences are positive \eqref{eqn:equiv_monotonicity} follows.
\end{proof}


\setlength{\parskip}{0pt plus 1pt}
\bibliographystyle{alpha}
\bibliography{Yukawa-refs}

\end{document}